\documentclass[hidelinks,onefignum,onetabnum]{siamart251216}

\usepackage{comment} % Load the package

\usepackage{lipsum}
\usepackage{amsfonts}
\usepackage{graphicx}
\usepackage{epstopdf}
\usepackage{algorithmic}
\ifpdf
  \DeclareGraphicsExtensions{.eps,.pdf,.png,.jpg}
\else
  \DeclareGraphicsExtensions{.eps}
\fi

\usepackage{amssymb,bm}
\usepackage{float}
\usepackage{hyperref}
\usepackage{tikz}
\usetikzlibrary{calc,arrows.meta,positioning}

\newcommand{\one}{\boldsymbol{1}}

\newsiamremark{remark}{Remark}
\newsiamremark{hypothesis}{Hypothesis}
\crefname{hypothesis}{Hypothesis}{Hypotheses}
\newsiamthm{claim}{Claim}
\newsiamremark{fact}{Fact}
\crefname{fact}{Fact}{Facts}
\newtheorem{assumption}[theorem]{Assumption}

\headers{Pressure Monotonicity for Biot}{H. Gu, M. Cai, and J. Li}

\title{Pressure Monotonicity for a Rotation-Based Biot Discretization with Circumcentric Mass Lumping\thanks{Submitted to the editors DATE.
\funding{H. Gu is supported by the National Natural Science Foundation of China (Grant No.~12601782) and Shenzhen University of Information Technology (Grant No.~SUIT2025KJ008). 
J. Li was partially supported by the Shenzhen Sci-Tech Fund No. RCJC20200714114556020, Guangdong Basic and Applied Research Fund No. 2023B1515250005.
M. Cai is supported in part by the U. S. National Science Foundation under Grant No. DMS-2602787.  
Any opinions, findings, and conclusions or recommendations expressed in this material are those of the author(s) and do not necessarily reflect the views of the National Science Foundation.}}}

\author{
Huipeng Gu\thanks{School of Humanities and Fundamental Sciences, Shenzhen University of Information Technology, Shenzhen, 518172, China.
(\email{guhuipeng@suit-sz.edu.cn}).}
\and 
Mingchao Cai\thanks{Department of Mathematics, Morgan State University, Baltimore, MD 21251, USA.
(\email{mingchao.cai@morgan.edu}).}
\and 
Jingzhi Li\thanks{Department of Mathematics, Southern University of Science and Technology, Shenzhen, 518055, China.
(\email{li.jz@sustech.edu.cn}).}
}

\usepackage{amsopn}

\ifpdf
\hypersetup{
  pdftitle={Pressure Monotonicity for a Rotation-Based Biot Discretization with Circumcentric Mass Lumping},
  pdfauthor={H. Gu, M. Cai, and J. Li}
}
\fi

\begin{document}

\maketitle

% REQUIRED
\begin{abstract}
We establish cellwise pressure monotonicity for a mixed finite element discretization of the rotation-based Biot system. 
The spatial discretization is based on compatible discrete de Rham complexes, using standard lowest-order elements for rotation, displacement, Darcy flux, and pressure.
For the temporal discretization, we employ the backward Euler method.
Exactness of the discrete complexes reduces the mechanics contribution to the pressure Schur complement to a scaled pressure mass matrix, yielding a positive diagonal mechanics capacity.
For the Darcy part, positive circumcentric facet weights yield a weighted cell graph Laplacian after flux elimination, while the mass-lumped product remains uniformly equivalent to the standard Darcy product.
The diagonal mechanics capacity and the weighted cell graph Laplacian make the condensed pressure matrix Stieltjes.
Consequently, the pressure propagator is entrywise nonnegative, which yields a cellwise discrete maximum principle.
Numerical experiments confirm the predicted convergence rates and show that the proposed discretization suppresses spurious cellwise pressure oscillations.
\end{abstract}

% REQUIRED
\begin{keywords}
Rotation-based poroelasticity;
Discrete de Rham complex;
%Mechanics Schur complement;
Mass lumping;
%$M$-matrix;
Discrete maximum principle;
Pressure monotonicity.
\end{keywords}

% REQUIRED
\begin{MSCcodes}
65M60, 74F10, 35B50
\end{MSCcodes}

\section{Introduction}
Poroelasticity describes the interaction between the deformation of a porous skeleton and fluid flow. The classical theory of poroelasticity was established by Biot~\cite{biot1941general} and is widely used in geomechanics, energy engineering, and biomechanics. In practical applications, the computational domain may have a complex geometry and the material properties may be heterogeneous. Consequently, analytical solutions are generally limited to simplified settings, making numerical methods essential for practical problems.

The numerical approximation of poroelasticity is challenging in nearly
incompressible and low-permeability regimes, where volumetric locking
and spurious pressure oscillations may occur
\cite{phillips2009overcoming,yi2017study}. Mixed and total-pressure
formulations provide effective approaches to stability and locking
\cite{oyarzua2016locking}, but inf--sup stability alone does not ensure
cellwise pressure monotonicity. Previous studies have demonstrated
the loss of monotonicity for stable discretizations and related
pressure oscillations to the sign structure of the condensed pressure
operator \cite{rodrigo2016stability,favino2017accuracy,ferreira2026spurious}.
For the coupled Biot system, this structure depends on both the
mechanics-induced pressure capacity and the discrete Darcy operator.
A monotonicity analysis must therefore account for these two
contributions simultaneously.

In this work, we adopt the rotation-based Biot formulation and
compatible mixed finite element spaces of
Boon et al.~\cite{boon2023mixed}, together with circumcentric
mass lumping for the Raviart--Thomas Darcy flux
\cite{brezzi2006error,hu2017nonconforming} and backward Euler time
discretization. Our main contribution is a structural pressure
monotonicity theorem in two and three dimensions. For constant
coefficients and the mixed mechanics boundary conditions specified
below, relative exactness yields an exact reduction of the mechanics
pressure Schur complement to a scaled pressure mass matrix.
Consequently, the effective pressure capacity is the positive
diagonal matrix
$H_h=(c_0+\alpha^2/(\lambda+2\mu))M_p$.
Positive assembled circumcentric facet weights make the condensed
Darcy operator a weighted cell-graph Laplacian with boundary
contributions. Together, these properties make the pressure matrix
Stieltjes for every time step $\tau>0$, yielding order-preserving
pressure propagation and a cellwise discrete maximum principle for
the homogeneous update. Positivity also holds for forced updates
with nonnegative effective loads. The prescribed mechanics boundary
conditions are essential hypotheses of this result because they
provide the relative exactness and divergence surjectivity required
for the exact capacity identity. Under stronger quantitative
geometric conditions, and assuming the corresponding mixed
projection estimates and suitable solution regularity, we further
derive first-order convergence in space and time.

%The remainder of the paper is organized as follows.
%Section~\ref{sec:model} presents the rotation-based Biot formulation and its lowest-order mixed finite element discretization with a circumcentric mass-lumped Darcy metric.
%Section~\ref{sec:monotonicity} derives the mechanics pressure Schur complement, the cell-graph representation of the Darcy operator, and the pressure monotonicity results.
%Section~\ref{sec:convergence} establishes the first-order space--time error estimate and discusses the three-dimensional Darcy approximation.
%Section~\ref{sec:numerics} presents numerical experiments for convergence, pressure monotonicity, and the Barry--Mercer problem.
%Section~\ref{sec:conclusion} concludes the paper.

\begin{comment}
% The outline is not required, but we show an example here.
The paper is organized as follows.
The formulation and discretization are presented in Section~\ref{sec:model}.
The main pressure monotonicity results are established in Section~\ref{sec:monotonicity}, while the first-order error estimate is derived in Section~\ref{sec:convergence}.
The numerical results are reported in Section~\ref{sec:numerics}, and conclusions are drawn in Section~\ref{sec:conclusion}.
\end{comment}

The paper is organized as follows.
Section~\ref{sec:model} presents the formulation and discretization.
Sections~\ref{sec:monotonicity} and~\ref{sec:convergence} establish pressure monotonicity and the first-order error estimate, respectively.
Section~\ref{sec:numerics} presents numerical experiments to verify the proposed approach and theoretical analysis.
Section~\ref{sec:conclusion} concludes the paper.

\section{Mathematical model and discretization}
\label{sec:model}

\subsection{Notation}
\label{sec:notation}

% Domain and Time Setting
Let $\Omega\subset\mathbb{R}^d$, $d\in\{2,3\}$, be a bounded, connected, contractible Lipschitz domain, polygonal if $d=2$ and polyhedral if $d=3$, with boundary \(\partial\Omega\). 
Let \(\nu\) denote the outward unit normal.
Let \(T_f>0\) be the prescribed final time, and consider $t \in [0,T_f]$. 
% For a time-dependent quantity \(\chi\), its time derivative is denoted by \(\dot \chi:=\partial_t \chi\).

% Spaces
Let \(L^2(\Omega)\) denote the space of square-integrable scalar functions on \(\Omega\), endowed with the inner product \((\cdot,\cdot)\) and the corresponding norm \(\lVert\cdot\rVert\).
The standard Sobolev spaces are denoted by $H^s(\Omega)$, with norm $\lVert\cdot\rVert_{H^s(\Omega)}$.
The curl operator $\nabla\times$ is understood in the standard sense: in $d=2$, it denotes the scalar-to-vector and vector-to-scalar curl operators, while in $d=3$ it denotes the standard vector curl. Let \( k_d:=\binom{d}{2},\) so that $k_2=1$ and $k_3=3$. We define
\begin{align*}
H(\nabla\times;\Omega)
:=
\{
\rho\in [L^2(\Omega)]^{k_d}:
\nabla\times\rho\in [L^2(\Omega)]^d
\}, \\
H(\nabla\cdot;\Omega)
:=
\{
 z\in [L^2(\Omega)]^d:
\nabla\cdot z\in L^2(\Omega)
\},
\end{align*}
equipped with the graph norms
\begin{align*}
\lVert \rho\rVert_{H(\nabla\times;\Omega)}^2
:=
\lVert \rho\rVert^2
+
\lVert \nabla\times\rho\rVert^2, \quad
\lVert z \rVert_{H(\nabla\cdot;\Omega)}^2
:=
\lVert z \rVert^2
+
\lVert \nabla\cdot z \rVert^2.
\end{align*}
In particular, $H(\nabla\times;\Omega)=H^1(\Omega)$ when $d=2$.
For a relatively open boundary portion $\Gamma\subset\partial\Omega$, we introduce the subspaces with homogeneous boundary conditions
\begin{align*}
H_\Gamma(\nabla\times;\Omega)
:=
\left\{
\rho\in H(\nabla\times;\Omega):
\nu\times \rho=0\ \text{on }\Gamma
\right\},\\
H_\Gamma(\nabla\cdot;\Omega)
:=
\left\{
 z\in H(\nabla\cdot;\Omega):
 z\cdot\nu=0\ \text{on }\Gamma
\right\}.
\end{align*}

For nonnegative quantities \(a\) and \(b\), we write \( a\lesssim b \) if there exists a constant \(C>0\), independent of the mesh size \(h\), such that \(a\le Cb\). We write \( a\simeq b \) if both \(a\lesssim b\) and \(b\lesssim a\) hold. The hidden constants may depend on the domain, the shape-regularity constant of the mesh family, and the physical parameters.

Vector and matrix inequalities with respect to the symbols \(\le\) and \(\ge\) are understood componentwise. For symmetric matrices, \(A\succeq0\) and \(A\succ0\) denote the Loewner order, i.e., positive semidefiniteness and positive definiteness, respectively. 
The all-ones vector of appropriate dimension is denoted by \(\one\).
A nonnegative matrix \(G\in\mathbb{R}^{n\times n}\) is called \emph{row-stochastic} if \(G\one=\one\), and \emph{row-substochastic} if \(G\one\le\one\). 
We also use the following standard matrix classes:
\begin{itemize}
    \item A $Z$-matrix has nonpositive off-diagonal entries.
    \item A nonsingular $M$-matrix is a $Z$-matrix with nonnegative inverse.
    \item A Stieltjes matrix is a symmetric positive-definite $M$-matrix.
\end{itemize}

\subsection{Rotation-based Biot formulation}
\label{subsec:continuous-model}

We consider the rotation-based Biot formulation in \cite{boon2023mixed}. In addition to the primary unknowns of displacement \(u\), Darcy flux \(v\), and pressure \(p\), the rotation \(r=\mu\nabla\times u\) is introduced as an auxiliary unknown. The resulting system reads
\begin{subequations}
\label{eq:continuous}
\begin{align}
  \mu^{-1}r-\nabla\times u &= 0, 
  \label{eq:continuous-a}\\
  \nabla\times r -\nabla\bigl((\lambda+2\mu)\nabla\cdot u-\alpha p\bigr) &= f, 
  \label{eq:continuous-b}\\
  K^{-1}v+\nabla p &= 0, 
  \label{eq:continuous-c}\\
  c_0 \partial_t p+\alpha\nabla\cdot \partial_t{u}+\nabla\cdot v &=s. 
  \label{eq:continuous-d}
\end{align}
\end{subequations}
Here, $\mu>0$ and $\lambda>0$ are the constant Lamé parameters, $\alpha>0$ is the Biot-Willis coefficient, $c_0\ge0$ is the storage coefficient, and $K>0$ is the permeability coefficient. 
The functions $ f$ and $s$ denote the body force and fluid source term, respectively.

Let \( \partial\Omega=\overline{\Gamma_r}\cup\overline{\Gamma_u} \) be a disjoint decomposition into relatively open boundary portions, with \(\Gamma_u\neq\varnothing\). 
We allow \(\Gamma_r=\varnothing\); if \(\Gamma_r\neq\varnothing\), we assume that \(\Gamma_r\) is contractible.
We then impose the following homogeneous boundary conditions:
\begin{align*}
\nu\cdot u &=0,
& \quad
\nu\times r &=0
& & \quad\text{on }\Gamma_r, \\ % \label{eq:bc-r}\\
\nu\times u &=0,
& \quad
(\lambda+2\mu)\nabla\cdot u-\alpha p&=0
& & \quad\text{on }\Gamma_u, \\ % \label{eq:bc-u}\\
& & p &=0
& & \quad\text{on }\partial\Omega . % \label{eq:bc-p}
\end{align*}
In two dimensions, the condition \(\nu\times r=0\) on \(\Gamma_r\) is understood as \(r=0\) on \(\Gamma_r.\)
To ensure uniqueness of the displacement field, we assume that the mixed boundary conditions eliminate all nontrivial rigid body motions, i.e.,
$\nu \cdot \phi=0$ on $\Gamma_r$ and $\nu \times \phi=0$ on $\Gamma_u$
imply $\phi = 0$ for every $\phi\in\mathrm{RM}(\Omega)$,
where $\mathrm{RM}(\Omega)$ denotes the space of rigid body motions. 

The initial data are assumed to be compatible with the quasi-static constraints of the Biot system. 
More precisely, we prescribe
\[
p(\cdot,0)=p_0,
\qquad
u(\cdot,0)=u_0,
\qquad
r(\cdot,0)=\mu\nabla\times u_0,
\qquad
v(\cdot,0)=-K\nabla p_0,
\]
with \(u_0\) the equilibrium displacement associated with \(p_0\) and \(f(0)\).

We introduce the spaces 
\[
  R=H_{\Gamma_r}(\nabla\times;\Omega),\qquad
  U=H_{\Gamma_r}(\nabla\cdot;\Omega),\qquad
  V=H(\nabla\cdot;\Omega),\qquad
  P=L^2(\Omega).
\]
The first three spaces inherit the structure of the relative de Rham complex 
\[
H_{\Gamma_r}(\nabla\times;\Omega)
\xrightarrow{\nabla\times}
H_{\Gamma_r}(\nabla\cdot;\Omega)
\xrightarrow{\nabla\cdot}
L^2(\Omega),
\]
and the discrete spaces below are chosen to preserve this structure.
The continuous weak formulation reads: For almost every $t \in (0, T_f]$, find \( (r(t),u(t),v(t),p(t))\in R\times U\times V\times P \) such that, for all \((\rho, w, z,q)
\in R\times U\times V\times P\), 
\begin{subequations}
\label{eq:weak-form}
\begin{align}
  \mu^{-1}(r,\rho)
  -(u,\nabla\times\rho)
  &=0, \label{eq:weak-rotation}\\
  (\nabla\times r, w)
  +(\lambda+2\mu)(\nabla\cdot u,
      \nabla\cdot w)
  -\alpha(p,\nabla\cdot w)
  &=( f, w), \label{eq:weak-momentum}\\
  K^{-1}(v, z)
  -(p,\nabla\cdot z)
  &=0, \label{eq:weak-darcy}\\
  c_0(\partial_t p,q)
  +\alpha(\nabla\cdot \partial_t{u},q)
  +(\nabla\cdot v,q)
  &=(s,q). \label{eq:weak-mass}
\end{align}
\end{subequations}

\subsection{Discrete formulation}
\label{subsec:discrete-model}

The finite element discretization is designed to satisfy two key structural properties.
First, we employ lowest-order mixed finite elements so that the pressure is represented by cellwise constants and the Darcy flux by facet-based Raviart--Thomas degrees of freedom. 
With this choice of spaces, the Darcy flux inner product admits a circumcentric mass-lumped approximation, resulting in a diagonal flux mass matrix; see, e.g.,
\cite{brezzi2006error,hu2017nonconforming}. 
Second, the mechanics spaces are chosen as a compatible finite element subcomplex of the de Rham sequence, which provides the discrete exactness and stability properties required \cite{boon2023mixed}.

Let \(\{\mathcal T_h\}\) be a shape-regular family of conforming simplicial meshes of \(\Omega\) that resolves the boundary decomposition. We use the lowest-order conforming compatible family considered in \cite{boon2023mixed} and define
\[
R_h\times U_h\times V_h\times P_h
:=
\bigl(
\mathcal{R}_h
\times \mathrm{RT}_0(\mathcal{T}_h)
\times \mathrm{RT}_0(\mathcal{T}_h)
\times \mathrm{DG}_0(\mathcal{T}_h)
\bigr)
\cap
(R\times U\times V\times P),
\]
where
\[
\mathcal{R}_h
:=
\begin{cases}
L_1(\mathcal{T}_h), & d=2,\\[1mm]
N_0^{(1)}(\mathcal{T}_h), & d=3.
\end{cases}
\]
Thus, the rotation is approximated by continuous piecewise linear Lagrange elements in two dimensions and by lowest-order Nédélec elements of the first kind in three dimensions. The spaces $R_h$, $U_h$, and $P_h$ form the relevant part of a
compatible discrete relative de Rham complex,
\[
R_h
\xrightarrow{\nabla\times}
U_h
\xrightarrow{\nabla\cdot}
P_h .
\]
Under the standing topological assumptions on
$(\Omega,\Gamma_r)$, the relative discrete complex is exact at
$U_h$, while $\Gamma_u\neq\varnothing$ gives the surjectivity of
the discrete divergence. Hence,
%Under the assumptions on $(\Omega,\Gamma_r,\mathcal{T}_h)$ and \(\Gamma_u\neq\varnothing\), this discrete complex is exact in the sense that
\[
\nabla\times R_h
=
\ker\!\left(\nabla\cdot|_{U_h}\right),
\qquad
\nabla\cdot U_h=P_h .
\]
In addition, the standard uniform inf-sup stability of the $\mathrm{RT}_0$--$\mathrm{DG}_0$ pair with the corresponding boundary conditions gives 
\begin{equation}
\label{eq:uniform-inf-sup}
\sup_{0\ne w_h\in U_h}
\frac{(\nabla\cdot w_h,q_h)}
     {\|w_h\|_{H(\nabla\cdot;\Omega)}}
\ge \beta_u\|q_h\|,
\ 
\sup_{0\ne z_h\in V_h}
\frac{(\nabla\cdot z_h,q_h)}
     {\|z_h\|_{H(\nabla\cdot;\Omega)}}
\ge \beta_v\|q_h\|, \quad \forall\,q_h\in P_h,
\end{equation}
where \(\beta_u,\beta_v>0\) are independent of \(h\).

\begin{remark}
\label{rem:higher-order-mechanics}
The structural arguments below rely on the compatible exact sequence property, rather than on the use of a lowest order discretization for the mechanics problem.
In two dimensions, an alternative compatible sequence is
\[
L_2(\mathcal{T}_h)
\xrightarrow{\nabla\times}
\mathrm{BDM}_1(\mathcal{T}_h)
\xrightarrow{\nabla\cdot}
\mathrm{DG}_0(\mathcal{T}_h).
\]
The derivation of the mechanics Schur complement below depends only on this compatible structure and therefore extends unchanged to such choices.
\end{remark}

%To expose the monotonicity structure of the discrete pressure operator, we replace the standard Darcy bilinear form with a circumcentric mass-lumped bilinear form on the lowest-order Raviart--Thomas space.
% As shown below, this construction yields a diagonal Darcy flux matrix in the canonical facet basis of $\mathrm{RT}_0$.
%For each interior facet \(F\in\mathcal F_h^{i}\), fix an arbitrary unit normal \(\nu_F\). For each boundary facet \(F\in\mathcal F_h^\partial\), let \(T(F)\) denote its unique adjacent cell and set \(\nu_F:=\nu_{T(F),F}.\)

Let $\mathcal F_h=\mathcal F_h^{i}\cup\mathcal F_h^{\partial}$ denote the sets of interior and boundary facets of \(\mathcal T_h\).
For each cell \(T\in\mathcal T_h\) and each facet \(F\subset\partial T\), let \(\nu_{T,F}\) denote the outward unit normal to \(F\) from \(T\).
For each facet \(F\in\mathcal F_h\), we fix a unit normal \(\nu_F\). On interior facets the orientation is arbitrary, so that \(\nu_{T,F}=\pm\nu_F\) for the adjacent cells, while on boundary facets \(\nu_F\) is chosen as the outward normal.
The corresponding \(\mathrm{RT}_0\) flux degree of freedom is defined by
\begin{align}
\label{flux-equation}
\Phi_F( z_h)
:=
\int_F  z_h\cdot\nu_F\,ds.
\end{align}

For $T\in\mathcal T_h$ and $F\subset\partial T$, let $c_T$ and $c_F$ denote their circumcenters, respectively. We define the signed local circumcentric distance and the associated facet weight by
\[
d_{T,F}^{s}
:=
(c_F-c_T)\cdot\nu_{T,F},
\qquad
a_F
:=
\sum_{T': F\subset\partial T'}
\frac{d_{T',F}^{s}}{K|F|}.
\] 
Finally, we define the circumcentric mass-lumped Darcy bilinear form by
\begin{equation}
\label{eq:mass-lumped-darcy-form}
a_{v,h}^{L}(v_h,z_h)
:=
\sum_{F\in\mathcal F_h}
a_F\,\Phi_F(v_h)\Phi_F(z_h).
\end{equation}
In the canonical \(\mathrm{RT}_0\) facet basis dual to the flux degrees of freedom \(\{\Phi_F\}\), this form yields a diagonal Darcy matrix with diagonal entries \(a_F\).
We then impose the following positivity condition, which ensures that \(a_{v,h}^{L}(\cdot,\cdot)\) defines an inner product on \(V_h\).
\begin{assumption}[Geometric positivity]
\label{ass:geometric-positivity}
The circumcentric weights satisfy
\[
a_F>0 \qquad \forall F\in\mathcal F_h .
\]
A sufficient condition is that every simplex be well centered. 
\end{assumption}

The above assumption is sufficient for \(M\)-matrix argument. For the convergence analysis, we impose the following stronger uniform version of geometric positivity:
\begin{equation}
\label{eq:uniform-well-centered}
  c_*h_T
  \le d_{T,F}^{s}
  \le C_*h_T
  \qquad
  \forall\,T\in\mathcal T_h,\quad F\subset\partial T,
\end{equation}
where \(h_T:=\operatorname{diam}(T)\), and \(c_*,C_*>0\) are independent of \(h\). 

For a simplex \(T\), the segment joining its circumcenter \(\boldsymbol c_T\) to the circumcenter \(\boldsymbol c_F\) of a facet \(F\subset\partial T\) is orthogonal to \(F\). Hence
\[
|d_{T,F}^{s}|
=
\operatorname{dist}(\boldsymbol c_T,F)
\le R_T
\lesssim h_T,
\]
where \(R_T\) is the circumradius of \(T\). 
Thus, the upper bound in \eqref{eq:uniform-well-centered} follows from shape regularity. 
The lower bound prevents the circumcenter from approaching any facet as the mesh is refined.
Under this geometric condition, the following lemma establishes the uniform equivalence between the circumcentric mass-lumped Darcy bilinear form and the \(L^2\) inner product on \(V_h\).

\begin{lemma}
\label{lem:lumped-equivalence}
Assume that the mesh family is shape regular and satisfies \eqref{eq:uniform-well-centered}. There exist constants \(c_L,C_L>0\) independent of \(h\), such that
\begin{equation}
\label{eq:lumped-equivalence}
  c_LK^{-1}\lVert z_h\rVert^2
  \le
  a_{v,h}^{L}( z_h, z_h)
  \le
  C_LK^{-1}\lVert z_h\rVert^2
  \qquad
  \forall\, z_h\in V_h.
\end{equation}
Moreover,
\begin{equation}
\label{eq:lumped-continuity}
  \bigl|a_{v,h}^{L}(w_h, z_h)\bigr|
  \le
  C_LK^{-1}
  \lVert w_h\rVert
  \lVert z_h\rVert
  \qquad
  \forall\,w_h, z_h\in V_h.
\end{equation}
\end{lemma}

\begin{proof}
For \(F\subset\partial T\), we define the local outward normal flux by
\[
  \Phi_{T,F}( z_h)
  :=
  \int_F
  \left. z_h\right|_T
  \cdot \nu_{T,F}\, ds.
\]
Since \(\nu_{T,F}\) and \(\nu_F\) are unit normals to the same facet, they differ only by a sign.
\[
  \Phi_{T,F}(z_h)=\pm\Phi_F(z_h).
\]
Thus, the orientation signs cancel in the product.
\[
\Phi_{T,F}(w_h)\Phi_{T,F}(z_h)
=
\Phi_F(w_h)\Phi_F(z_h) \qquad F\subset \partial T.
\]
Substituting the definition of \(a_F\) into \eqref{eq:mass-lumped-darcy-form} gives
\begin{align}
\label{eq:local-lumped-darcy-form}
a_{v,h}^{L}(w_h,z_h)
&=
K^{-1}
\sum_{F\in\mathcal F_h}
\sum_{T: F\subset \partial T}
\frac{d_{T,F}^{s}}{|F|}
\Phi_{T,F}(w_h)\Phi_{T,F}(z_h)
\notag \\
&=
K^{-1}
\sum_{T\in\mathcal T_h}
\sum_{F\subset\partial T}
\frac{d_{T,F}^{s}}{|F|}
\Phi_{T,F}(w_h)\Phi_{T,F}(z_h).
\end{align}
The shape regularity and \eqref{eq:uniform-well-centered} imply
\(
  |F|\simeq h_T^{d-1},
  d_{T,F}^{s}\simeq h_T.
\)
Thus, we have
\begin{equation}
\label{eq:local-lumped-weight-scaling}
  \frac{d_{T,F}^{s}}{|F|}
  \simeq h_T^{2-d},
\end{equation}
with constants independent of \(T\), \(F\), and \(h\).

Standard properties of the \(\mathrm{RT}_0\) element, including unisolvence of the facet flux degrees of freedom, preservation of integral normal fluxes under the affine Piola transformation, and the corresponding scaling estimates, are given in \cite[Lemmas~2.1.6, 2.1.8 and Proposition~2.3.4]{boffi2013mixed}. On the reference simplex, unisolvence implies that the Euclidean norm of the facet flux vector defines a norm on \(\mathrm{RT}_0\), which is equivalent to the \(L^2\) norm by finite dimensionality. Consequently,
\[
  \lVert z_h\rVert_{L^2(T)}^2
  \simeq
  h_T^{2-d}
  \sum_{F\subset\partial T}
  \left( \Phi_{T,F}(z_h) \right)^2
  \qquad
  \forall\,\left. z_h\right|_T\in\mathrm{RT}_0(T).
\]
Combining this estimate with \eqref{eq:local-lumped-weight-scaling} gives
\[
  \sum_{F\subset\partial T}
  \frac{d_{T,F}^{s}}{|F|}
  \left(\Phi_{T,F}( z_h)\right)^2
  \simeq
  \lVert z_h\rVert_{L^2(T)}^2.
\]
Applying the above local estimate to \eqref{eq:local-lumped-darcy-form} with \(w_h=z_h\) yields
\[
  a_{v,h}^{L}(z_h, z_h)
  \simeq
  K^{-1}
  \sum_{T\in\mathcal T_h}
  \lVert z_h\rVert_{L^2(T)}^2
  =
  K^{-1}
  \lVert z_h\rVert^2.
\]
This proves \eqref{eq:lumped-equivalence}. 
Since \(a_{v,h}^{L}(\cdot,\cdot)\) is an inner product, the Cauchy--Schwarz inequality and the upper bound in \eqref{eq:lumped-equivalence} give
\[
  \bigl|a_{v,h}^{L}(w_h,z_h)\bigr|
  \le
  a_{v,h}^{L}(w_h,w_h)^{1/2}
  a_{v,h}^{L}(z_h,z_h)^{1/2}
  \le
  C_LK^{-1}
  \lVert w_h\rVert
  \lVert z_h\rVert,
\]
which proves \eqref{eq:lumped-continuity}.
\end{proof}

For the temporal discretization, let
\( 0=t_0<t_1<\cdots<t_N=T_f \)
be a uniform partition of $[0,T_f]$ with time step $\tau=T_f/N$. 
For any time-dependent quantity $\chi$, we write $\chi^n:=\chi(t_n)$. 
We employ backward Euler time stepping and, for a discrete sequence $\{\chi^n\}_{n=0}^{N}$, define
\[
\delta_t \chi^n
:=
\frac{\chi^n-\chi^{n-1}}{\tau},
\qquad
1\le n\le N.
\]

The initial data are assumed to be compatible with the discrete quasi-static constraints.
The fully discrete scheme reads as follows. For each \(n=1,\ldots,N\), find
\(
  (r_h^n,u_h^n,v_h^n,p_h^n)
  \in R_h\times U_h\times V_h\times P_h,
\)
such that, for all
\((\rho_h, w_h, z_h,q_h)
\in R_h\times U_h\times V_h\times P_h\),
\begin{subequations}
\label{eq:discrete-weak-form}
\begin{align}
  \mu^{-1}(r_h^n,\rho_h)
  -(u_h^n,\nabla\times\rho_h)
  & = 0, \label{eq:discrete-weak-rotation}\\
  (\nabla\times r_h^n, w_h)
  +(\lambda+2\mu)(\nabla\cdot u_h^n,
      \nabla\cdot w_h)
  -\alpha(p_h^n,\nabla\cdot w_h)
  & = (f^n, w_h),
  \label{eq:discrete-weak-momentum}\\
  a_{v,h}^{L}(v_h^n, z_h)
  -(p_h^n,\nabla\cdot z_h)
  & = 0,
  \label{eq:discrete-weak-darcy}\\
  c_0( \delta_t p_h^n,q_h)
  +\alpha(\nabla\cdot \delta_t u_h^n,q_h)
  + (\nabla\cdot v_h^n,q_h)
  & = (s^n,q_h).
  \label{eq:discrete-weak-mass}
\end{align}
\end{subequations}

We next record the algebraic form of
\eqref{eq:discrete-weak-form}. Choose bases
$\{\psi_i^r\}_{i=1}^{N_r}$,
$\{\psi_j^u\}_{j=1}^{N_u}$,
$\{\psi_F^v\}_{F\in\mathcal F_h}$, and
$\{\psi_T^p\}_{T\in\mathcal T_h}$
of $R_h$, $U_h$, $V_h$, and $P_h$, respectively. 
Accordingly, we write
\[
r_h^n=\sum_{i=1}^{N_r} r_i^n \psi_i^r,
\quad
u_h^n=\sum_{j=1}^{N_u} u_j^n \psi_j^u,
\quad
v_h^n=\sum_{F\in\mathcal F_h}v_F^n \psi_F^v,
\quad
p_h^n=\sum_{T\in\mathcal T_h}p_T^n\psi_T^p,
\]
and denote the corresponding coefficient vectors by $\mathbf r^n$, $\mathbf u^n$, $\mathbf v^n$, and $\mathbf p^n$.
Let $M_r$ and $M_p$ denote the $L^2$ mass matrices on $R_h$ and $P_h$, respectively:
\[
(M_r)_{ki}
:=
( \psi_i^r, \psi_k^r),
\qquad
(M_p)_{ST}
:=
(\psi_T^p,\psi_S^p).
\]
The mass matrix \(M_r\) is symmetric positive definite and therefore invertible. Since $\psi_T^p=\chi_T$ in $P_h$, the entries of $\mathbf p^n$ are precisely the cellwise pressure values, yielding
\[
M_p = \operatorname{diag}\bigl(|T|\bigr)_{T\in\mathcal T_h}.
\]
Define the discrete coupling matrices by
\begin{align*}
(C_h)_{ji}
& :=
(\nabla\times \psi_i^r,
  \psi_j^u),
\qquad
& (B_u)_{Tj}
& :=
(\psi_T^p,\nabla\cdot \psi_j^u),
\\
(B_v)_{TF}
& :=
(\psi_T^p,\nabla\cdot \psi_F^v),
\qquad
& (D_u)_{jk}
& :=
(\nabla\cdot \psi_k^u,
 \nabla\cdot \psi_j^u).
\end{align*}
For the Darcy flux space \(V_h\), we choose the canonical
\(\mathrm{RT}_0\) facet basis dual to the flux degrees of freedom
\(\{\Phi_F\}_{F\in\mathcal F_h}\). Thus, for \(F,G\in\mathcal F_h\),
\begin{equation}
\label{RT0-flux}
\Phi_F(\psi_G^v)
= 
\int_F \psi_G^v \cdot \nu_F \,ds
=
\delta_{FG}
=
\begin{cases}
1, & F=G,\\
0, & F\neq G.
\end{cases}
\end{equation}
%It follows that
%\( v_F^n=\Phi_F(v_h^n), \)
%so that the coefficient \(v_F^n\) is precisely the total discrete normal flux through the facet \(F\). 
Writing
\( v_h^n=\sum_{G\in\mathcal F_h}v_G^n\psi_G^v, \)
we obtain
\( v_F^n=\Phi_F(v_h^n), \)
so that \(v_F^n\) is precisely the signed integral normal flux of
\(v_h^n\) across \(F\), with respect to the chosen orientation \(\nu_F\).
Using \eqref{eq:mass-lumped-darcy-form}, the entries of the Darcy matrix satisfy, for \(F,G\in\mathcal F_h\),
\[
\begin{aligned}
(A_{v,h}^{L})_{FG}
:=
a_{v,h}^{L}(\psi_G^v,\psi_F^v) 
=
\sum_{E\in\mathcal F_h}
a_E\,\Phi_E(\psi_G^v)\Phi_E(\psi_F^v) 
=
\sum_{E\in\mathcal F_h}
a_E\,\delta_{EG}\delta_{EF}
=
a_F\delta_{FG}.
\end{aligned}
\]
Hence, under Assumption~\ref{ass:geometric-positivity}, all diagonal entries are positive, so \(A_{v,h}^{L}\) is symmetric positive definite and therefore invertible. In particular,
\[
A_{v,h}^{L}
=
\operatorname{diag}(a_F)_{F\in\mathcal F_h}.
\]

The load vectors are defined by
\[
(\mathbf f^n)_j
:=
( f^n, \psi_j^u),
\qquad
(\mathbf s^n)_T
:=
(s^n,\psi_T^p).
\]

The algebraic system corresponding to
\eqref{eq:discrete-weak-form} is
\begin{equation}
\label{eq:fully-discrete-matrix}
\begin{bmatrix}
\mu^{-1}M_r & -C_h^\top & 0 & 0
\\
C_h & (\lambda+2\mu)D_u & 0 & -\alpha B_u^\top
\\
0 & 0 & A_{v,h}^{L} & -B_v^\top
\\
0 & \alpha B_u & \tau B_v & c_0M_p
\end{bmatrix}
\begin{bmatrix}
\mathbf r^n\\
\mathbf u^n\\
\mathbf v^n\\
\mathbf p^n
\end{bmatrix}
=
\begin{bmatrix}
0\\
\mathbf f^n\\
0 \\
c_0M_p\mathbf p^{n-1}
+\alpha B_u\mathbf u^{n-1}
+\tau\mathbf s^n
\end{bmatrix}.
\end{equation}

Eliminating the rotation variable gives the discrete mechanics matrix
\begin{equation}
\label{eq:mechanics-matrix}
A_{u,h}
:=
\mu C_hM_r^{-1}C_h^\top
+
(\lambda+2\mu)D_u .
\end{equation}
As shown in the following Subsection~\ref{subsec:mechanics-capacity}, \(A_{u,h}\) is symmetric positive definite and hence invertible under the standing assumptions.
Define
\begin{equation*}
%\label{eq:pressure-operators}
H_h
:=
c_0M_p
+\alpha^2B_uA_{u,h}^{-1}B_u^\top,
\qquad
G_h
:=
B_v(A_{v,h}^{L})^{-1}B_v^\top .
\end{equation*}
Eliminating the displacement and Darcy flux from \eqref{eq:fully-discrete-matrix} yields
\begin{equation*}
%\label{eq:pressure-schur}
(H_h+\tau G_h)\mathbf p^n
=
c_0M_p\mathbf p^{n-1}
+\alpha B_u\mathbf u^{n-1}
+\tau\mathbf s^n
-\alpha B_uA_{u,h}^{-1}\mathbf f^n .
\end{equation*}
Using
\(
\mathbf u^{n-1}
=
A_{u,h}^{-1}
(
\mathbf f^{n-1}
+\alpha B_u^\top\mathbf p^{n-1}
),
\)
we obtain the one-step pressure equation
\begin{equation*}
%\label{eq:pressure-update}
K_h\mathbf p^n
=
H_h\mathbf p^{n-1}
+\boldsymbol\ell^n,
\quad
K_h:=H_h+\tau G_h,
\quad 
\boldsymbol\ell^n
:=
\alpha B_uA_{u,h}^{-1}
(\mathbf f^{n-1}-\mathbf f^n)
+\tau\mathbf s^n.
\end{equation*}
Equivalently,
\begin{equation}
\label{eq:pressure-propagation}
\mathbf p^n
=
S_h\mathbf p^{n-1}
+
K_h^{-1}\boldsymbol\ell^n,
\qquad
S_h:=K_h^{-1}H_h .
\end{equation}
The matrix $S_h$ governs the homogeneous one-step pressure propagation, while $K_h^{-1}\boldsymbol\ell^n$ represents the additive response to external loading and source terms. % Accordingly, the monotonicity analysis below focuses on $S_h$.

\section{Pressure monotonicity}
\label{sec:monotonicity}

Pressure monotonicity is an entrywise algebraic property and is not implied by energy stability. 
As emphasized in \cite{ferreira2026spurious}, monotonicity requires additional algebraic structure to ensure a discrete maximum principle and to exclude local overshoots and undershoots. 
This motivates the use of the $M$-matrix framework for the pressure equation.
The following result is a standard consequence of the theory of nonsingular $M$-matrices \cite{windisch1989mmatrices}.
\begin{lemma}
\label{lem:stieltjes}
Let \(A\in\mathbb R^{m\times m}\) be a symmetric positive-definite \(Z\)-matrix. Then \(A\) is a Stieltjes matrix, and hence a nonsingular \(M\)-matrix. In particular,
\[
A^{-1}\ge0.
\]
Moreover, if $H\ge0$, then
\(
A^{-1}H\ge0.
\)
\end{lemma}

If \(H_h\) is positive diagonal and \(G_h\) is a symmetric
positive-semidefinite \(Z\)-matrix, then
\(K_h=H_h+\tau G_h\) is a symmetric positive-definite \(Z\)-matrix
and hence a Stieltjes matrix. Lemma~\ref{lem:stieltjes} therefore gives
\(K_h^{-1}\ge0\), and since \(H_h\ge0\),
\(
S_h=K_h^{-1}H_h\ge0.
\)

\subsection{Mechanics capacity}
\label{subsec:mechanics-capacity}

We first examine the mechanics contribution to the matrix $H_h = c_0M_p
+\alpha^2B_uA_{u,h}^{-1}B_u^\top$. Although
\(
B_uA_{u,h}^{-1}B_u^\top
\)
contains the inverse of the global mechanics matrix, the compatible discrete de Rham structure yields an exact reduction of this Schur complement to a scalar multiple of the pressure mass matrix. This identity is the key to showing that $H_h$ is positive diagonal.

Define the $L^2$-orthogonal complement of the discrete curl space by
\[
W_h := (\nabla\times R_h)^{\perp_{L^2}}
= \{ w_h\in U_h: (w_h,\nabla\times \rho_h)=0 \quad\forall\, \rho_h\in R_h \}.
\]
By the definition of $C_h$, if $\mathbf w$ is the coefficient vector of
$w_h\in U_h$, we have
\[
C_h^\top\mathbf w=0
\quad\Longleftrightarrow\quad
w_h\in W_h.
\]
The definition of \(W_h\) gives the orthogonal decomposition
\(
U_h= (\nabla\times R_h) \oplus W_h.
\)

Moreover, the restriction
\(
\nabla\cdot:W_h\longrightarrow P_h
\)
is an isomorphism. Indeed, by the discrete de Rham complex,
\(
\ker(\nabla\cdot|_{U_h})=\nabla\times R_h,
\)
and hence
\[
W_h\cap\ker(\nabla\cdot|_{U_h})
=
W_h\cap(\nabla\times R_h)
=
\{0\},
\]
which proves injectivity.
To prove surjectivity, let $q_h\in P_h$. By the discrete de Rham complex, we have
\(
\nabla\cdot U_h=P_h.
\)
There exists $u_h\in U_h$ such that
\[
\nabla\cdot u_h=q_h.
\]
Using the above decomposition, write
\[
u_h
=
\nabla\times \rho_h+w_h,
\qquad
 \rho_h\in R_h,
\quad
w_h\in W_h.
\]
Since
\(
\nabla\cdot(\nabla\times \rho_h)=0,
\)
it follows that
\(
\nabla\cdot w_h=q_h,
\)
which proves bijective.

This decomposition also shows immediately from \eqref{eq:mechanics-matrix} that \(A_{u,h}\) is positive definite.
Let \(w_h\in U_h\) be arbitrary, with coefficient vector
\(\mathbf w\in\mathbb R^{N_u}\). If
\[
\mathbf w^\top A_{u,h}\mathbf w
=
\mu(C_h^\top\mathbf w)^\top M_r^{-1}(C_h^\top\mathbf w)
+
(\lambda+2\mu) \lVert \nabla\cdot w_h \rVert_0^2
=
0,
\]
then, we have
\[
C_h^\top\mathbf w=0,
\qquad
\nabla\cdot w_h=0.
\]
The first relation gives $w_h\in W_h$. The second relation,
together with the injectivity of
\(
\nabla\cdot|_{W_h},
\)
implies $w_h=0$. Therefore
\(
A_{u,h}\succ0,
\)
and in particular $A_{u,h}$ is invertible.

\begin{theorem}
\label{thm:mechanics-schur}
The mechanics pressure Schur complement satisfies
\begin{equation}
\label{eq:mechanics-schur}
B_uA_{u,h}^{-1}B_u^\top
=
\frac{1}{\lambda+2\mu}M_p .
\end{equation}
Consequently,
\begin{equation}
\label{eq:mechanics-capacity}
H_h
=
\bigg(
c_0+\frac{\alpha^2}{\lambda+2\mu}
\bigg)M_p .
\end{equation}
In particular, since
\(
M_p=\operatorname{diag}\bigl(|T|\bigr)_{T\in\mathcal T_h},
\)
the matrix $H_h$ is positive diagonal.
\end{theorem}

\begin{proof}
Let $q_h\in P_h$ be arbitrary, and let $\mathbf q$ denote its coefficient vector in the chosen pressure basis. 
By the isomorphism of the restriction
\(
\nabla\cdot:W_h\longrightarrow P_h,
\) 
there exists a unique $w_q\in W_h$ such that
\begin{equation}
\label{eq:wq-div}
\nabla\cdot w_q
=
\frac{1}{\lambda+2\mu}q_h .
\end{equation}
Let $\mathbf w_q$ denote its coefficient vector. Since
$w_q\in W_h$, we have
\(
C_h^\top\mathbf w_q=0.
\)
For an arbitrary $ z_h\in U_h$, with coefficient vector $\mathbf z$, we then have
\[
\begin{aligned}
\mathbf z^\top A_{u,h}\mathbf w_q
&=
\mu(C_h^\top\mathbf z)^\top
M_r^{-1}(C_h^\top\mathbf w_q)
+
(\lambda+2\mu)
(\nabla\cdot w_q,\nabla\cdot z_h)
\\
&=
(\lambda+2\mu)
\bigg(\frac{1}{\lambda+2\mu}q_h,\nabla\cdot z_h\bigg)
=
(q_h,\nabla\cdot z_h)
=
\mathbf z^\top B_u^\top\mathbf q.
\end{aligned}
\]
Since $\mathbf z$ is arbitrary, we derive
\(
A_{u,h}\mathbf w_q=B_u^\top\mathbf q.
\)
By the invertibility of $A_{u,h}$, we obtain
\[
\mathbf w_q=A_{u,h}^{-1}B_u^\top\mathbf q.
\]
For each pressure basis function \(\psi_T^p\), we apply \eqref{eq:wq-div} to the definition of $B_u$.
\[
(B_u\mathbf w_q)_T
=
(\psi_T^p,\nabla\cdot w_q)
=
\frac{1}{\lambda+2\mu}(\psi_T^p,q_h)
=
\frac{1}{\lambda+2\mu}(M_p\mathbf q)_T.
\]
Therefore, we obtain
\[
B_uA_{u,h}^{-1}B_u^\top\mathbf q
=
B_u\mathbf w_q
=
\frac{1}{\lambda+2\mu}M_p\mathbf q.
\]
Since $\mathbf q$ is arbitrary, we can deduce \eqref{eq:mechanics-schur}.
Substituting this identity into
\(
H_h
=
c_0M_p+\alpha^2B_uA_{u,h}^{-1}B_u^\top
\)
gives the desired result \eqref{eq:mechanics-capacity}.
\end{proof}

\subsection{Mass-lumped Darcy operator}
\label{subsec:mass-lumped-darcy}

We next turn to the Darcy pressure operator $G_h= B_v(A_{v,h}^{L})^{-1}B_v^\top$.
By the circumcentric mass-lumped construction \eqref{eq:mass-lumped-darcy-form}, \(A_{v,h}^{L}\) is diagonal, while Assumption~\ref{ass:geometric-positivity} ensures that all its diagonal entries are positive. 
The columns of \(B_v\) are the oriented cell--facet incidence vectors described below.

For each interior facet \(F\in\mathcal F_h^i\), let
\(T^+(F)\) and \(T^-(F)\) denote its two adjacent cells, labeled such that
\[
\nu_{T^+(F),F}=\nu_F,
\qquad
\nu_{T^-(F),F}=-\nu_F.
\]
For boundary facets, we retain the notation \(T(F)\) introduced in Section \ref{subsec:discrete-model}. For each \(F\in\mathcal F_h\), define the cell--facet incidence vector \(\boldsymbol d_F\in\mathbb R^{N_p}\) by
\begin{equation}
(\boldsymbol d_F)_T
:=
\begin{cases}
  1, & T=T^+(F), \qquad F\in\mathcal F_h^{i},\\
 -1, & T=T^-(F), \qquad F\in\mathcal F_h^{i},\\
  1, & T=T(F),   \qquad \ \ F\in\mathcal F_h^\partial,\\
  0, & \text{otherwise}. 
\end{cases} 
\label{def-df}
\end{equation}
Thus, an interior facet has two incidence entries \(1\) and \(-1\), whereas a boundary facet has a single entry \(1\).
This is the standard cell--facet incidence structure \cite{alonsorodriguez2018graph}.

%The following proposition makes explicit the graph structure of the mass-lumped Darcy pressure operator.

\begin{proposition}
\label{prop:darcy-graph}
Under Assumption~\ref{ass:geometric-positivity}, the Darcy pressure operator satisfies
\begin{equation}
\label{eq:darcy-graph}
G_h
=
\sum_{F\in\mathcal F_h}
a_F^{-1}\boldsymbol d_F\boldsymbol d_F^\top .
\end{equation}
Consequently, $G_h$ is a symmetric positive-definite $Z$-matrix and
\(
G_h\one\ge \boldsymbol 0.
\)
\end{proposition}

\begin{proof}
Since $\psi_T^p=\chi_T$, the divergence theorem gives 
\[
(B_v)_{TF} 
=
(\psi_T^p,\nabla\cdot \psi_F^v)
=
\int_T \nabla\cdot \psi_F^v\,dx
=
\int_{\partial T} \psi_F^v \cdot \nu_T \,ds
= 
\sum_{G \subset \partial T} \int_{G} \psi_F^v \cdot \nu_{T,G} \,ds.
\]
Recalling the integral-flux normalization of the \(\mathrm{RT}_0\) facet basis from \eqref{flux-equation} and \eqref{RT0-flux},
\(
\Phi_G(\psi_F^v) = \delta_{FG},
\)
and using the orientation relation
\[
\nu_{T,G}
=
(\boldsymbol d_G)_T\,\nu_G
\qquad
\text{for }G\subset\partial T,
\]
we obtain
\[
\begin{aligned}
(B_v)_{TF}
&=
\sum_{G\subset\partial T}
(\boldsymbol d_G)_T
\int_G \psi_F^v\cdot\nu_G\,ds 
=
\sum_{G\subset\partial T}
(\boldsymbol d_G)_T
\Phi_G(\psi_F^v) 
=
(\boldsymbol d_F)_T.
\end{aligned}
\]
Hence the \(F\)-th column of \(B_v\) is precisely \(\boldsymbol d_F\).
Since
\(
A_{v,h}^{L}
=
\operatorname{diag}(a_F)_{F\in\mathcal F_h},
\)
we obtain
\begin{align*}
G_h
&=
B_v(A_{v,h}^{L})^{-1}B_v^\top
=
\sum_{F\in\mathcal F_h}
a_F^{-1}\boldsymbol d_F\boldsymbol d_F^\top, %\label{def-Gh}
\end{align*}
which proves \eqref{eq:darcy-graph}.

For any coefficient vector $\boldsymbol x=(x_T)_{T\in\mathcal T_h}\in\mathbb R^{N_p}$,
the definition of $\boldsymbol d_F$ gives
\[
\boldsymbol d_F^\top\boldsymbol x
=
\begin{cases}
x_{T^+(F)}-x_{T^-(F)},
& F\in\mathcal F_h^i,\\[1mm]
x_{T(F)},
& F\in\mathcal F_h^\partial.
\end{cases}
\]
Then, we obtain
\[
\begin{aligned}
\boldsymbol x^\top G_h\boldsymbol x
=
\sum_{F\in\mathcal F_h}
a_F^{-1}
(\boldsymbol d_F^\top\boldsymbol x)^2
=
\sum_{F\in\mathcal F_h^i}
a_F^{-1}
\bigl(x_{T^+(F)}-x_{T^-(F)}\bigr)^2
+
\sum_{F\in\mathcal F_h^\partial}
a_F^{-1}x_{T(F)}^2 .
\end{aligned}
\]
By Assumption~\ref{ass:geometric-positivity}, $a_F>0$ for every
$F\in\mathcal F_h$, and therefore
\(
\boldsymbol x^\top G_h\boldsymbol x\ge0.
\)
Thus, $G_h$ is symmetric positive semidefinite.

Moreover, if $\boldsymbol x^\top G_h\boldsymbol x=0$, then the following relations hold
\[
\begin{cases}
x_{T^+(F)}=x_{T^-(F)},
& \forall\,F\in\mathcal F_h^i,\\[1mm]
x_{T(F)}=0,
& \forall\,F\in\mathcal F_h^\partial.
\end{cases}
\]
Since the cell-adjacency graph is connected, any two cells
$T,T'\in\mathcal T_h$ can be joined by a sequence of cells
\(
T=T_0,T_1,\ldots,T_m=T',
\)
such that each pair $T_{k-1}$ and $T_k$ shares an interior facet.
The first relation yields
\[
x_{T_0}=x_{T_1}=\cdots=x_{T_m},
\]
and hence $x_T$ has the same value for all
$T\in\mathcal T_h$. On the other hand, for any boundary facet
$F\in\mathcal F_h^\partial$, the second relation gives
\(
x_{T(F)}=0.
\)
Thus, the common cell value must be zero, so that
$\boldsymbol x= \boldsymbol0$ and $G_h$ is positive definite.

For two distinct cells \(T,T'\in\mathcal T_h\), only an interior facet
shared by \(T\) and \(T'\) can contribute to the off-diagonal entry
\((G_h)_{TT'}\). Thus, by the definitions \eqref{def-df} and \eqref{eq:darcy-graph}
\[
(G_h)_{TT'}
=
\begin{cases}
-a_F^{-1},
& \text{if } T \text{ and } T' \text{ share } F\in\mathcal F_h^i,\\[1mm]
0,
& \text{otherwise}.
\end{cases}
\]
Since $a_F>0$, all off-diagonal entries of $G_h$ are nonpositive. Hence, $G_h$ is a $Z$-matrix.

For an interior facet $F\in\mathcal F_h^i$, the two nonzero components of $\boldsymbol d_F$ have opposite signs, and hence
\(
\boldsymbol d_F^\top\one=0.
\)
For a boundary facet $F\in\mathcal F_h^\partial$, the vector
$\boldsymbol d_F$ has a single nonzero component equal to $1$ at the cell $T(F)$, and hence
\(
\boldsymbol d_F^\top\one = 1.
\)
Consequently, for every $T\in\mathcal T_h$, we arrive at
\[
(G_h\one)_T
=
\left(\sum_{F\in\mathcal F_h}
a_F^{-1}\boldsymbol d_F
(\boldsymbol d_F^\top\one)
\right)_T
= 
\sum_{\substack{
F\in\mathcal F_h^\partial\\
F\subset\partial T}}
a_F^{-1}
\ge0.
\]
Hence
\(
G_h\one\ge \boldsymbol 0.
\) 
The proof is complete.
\end{proof}

\subsection{Pressure propagation}
\label{subsec:pressure-propagation}

Set \(\ell^n=0\) in \eqref{eq:pressure-propagation} to isolate the homogeneous pressure evolution from external loading and source terms:
\begin{equation}
\label{eq:pressure-step}
K_h\mathbf p^n=H_h\mathbf p^{n-1},
\qquad
\mathbf p^n=S_h\mathbf p^{n-1},
\qquad
S_h:=K_h^{-1}H_h.
\end{equation}

Since \(P_h=\mathrm{DG}_0\), the coefficients \(p_T^n\) are the cellwise pressure values. The prescribed boundary pressure is zero.
Accordingly, the relevant comparison interval is determined by the cellwise pressure values at the previous time level together with the boundary value zero. Define
\[
    p_{\min}^{n-1}
    :=
    \min\big\{
        0,\,
        \min_{T\in\mathcal T_h} p_T^{n-1}
    \big\},
    \qquad
    p_{\max}^{n-1}
    :=
    \max\big\{
        0,\,
        \max_{T\in\mathcal T_h} p_T^{n-1}
    \big\}.
\]
The cellwise discrete maximum principle takes the form
\begin{equation}
\label{eq:cellwise-dmp}
    p_{\min}^{n-1}
    \le
    p_T^n
    \le
    p_{\max}^{n-1}
    \qquad
    \forall T\in\mathcal T_h .
\end{equation}
Thus, the pressure update cannot create a new cellwise overshoot or undershoot outside the range determined by the previous pressure values and the prescribed boundary value zero. To establish \eqref{eq:cellwise-dmp}, it is sufficient to show that the pressure propagator \(S_h\) is entrywise nonnegative and row-substochastic. 
The following theorem derives these properties from the algebraic structure of \(H_h\) and \(G_h\).

\begin{theorem}[Pressure monotonicity]
\label{thm:monotone-pressure-propagation}
Under Assumption~\ref{ass:geometric-positivity}, the matrix
\(K_h\)
is a Stieltjes matrix for every \(\tau>0\). Moreover, the one-step
pressure propagator satisfies
\[
    S_h=K_h^{-1}H_h\ge0,
\]
and
\[
    \boldsymbol{0}\le S_h\one\le\one.
\]
Consequently, \(S_h\) is row-substochastic.
\end{theorem}

\begin{proof}
By Theorem~\ref{thm:mechanics-schur}, \(H_h\) is a positive diagonal matrix, while Proposition~\ref{prop:darcy-graph} shows that \(G_h\) is a symmetric positive-definite \(Z\)-matrix. It follows that
\[
    K_h=H_h+\tau G_h
\]
is symmetric positive definite for every \(\tau>0\). 

Moreover, since \(H_h\) is diagonal and \(G_h\) has nonpositive off-diagonal entries, \(K_h\) is also a \(Z\)-matrix. Hence \(K_h\) is a Stieltjes matrix.
By Lemma~\ref{lem:stieltjes}, we have
\[
    K_h^{-1}\ge0.
\]
Since \(H_h\) is positive diagonal, we immediately obtain
\[
    S_h
    =
    K_h^{-1}H_h
    \ge0.
\]

For the row-sum bound, Proposition~\ref{prop:darcy-graph} gives
\(
    G_h\one\ge0.
\)
Therefore,
\[
    K_h\one
    =
    H_h\one+\tau G_h\one
    \ge
    H_h\one.
\]
Since \(K_h^{-1}\ge0\), the order is preserved.
\[
    \one
    =
    K_h^{-1}K_h\one
    \ge
    K_h^{-1}H_h\one
    =
    S_h\one.
\]
On the other hand, \(S_h\ge0\) implies
\(
    S_h\one\ge \boldsymbol{0}.
\)
Thus,
\[
    \boldsymbol{0}\le S_h\one\le\one,
\]
which proves that \(S_h\) is row-substochastic.
\end{proof}

Theorem \ref{thm:monotone-pressure-propagation} immediately yields the desired cellwise discrete maximum principle.

\begin{corollary}[Discrete maximum principle]
\label{cor:cellwise-dmp}
Let \(\mathbf p^n\) be determined by the homogeneous pressure update
\eqref{eq:pressure-step}. Then \eqref{eq:cellwise-dmp} holds.
\end{corollary}

\begin{proof}
By definition of \(p_{\min}^{n-1}\) and \(p_{\max}^{n-1}\), we have
\[
    p_{\min}^{n-1}\one
    \le
    \mathbf p^{n-1}
    \le
    p_{\max}^{n-1}\one.
\]
Since \(S_h\ge0\), multiplication by \(S_h\) preserves the
componentwise inequalities:
\[
    p_{\min}^{n-1}S_h\one
    \le
    S_h\mathbf p^{n-1}
    \le
    p_{\max}^{n-1}S_h\one.
\]
Using
\[
    p_{\min}^{n-1}\le0,
    \qquad
    p_{\max}^{n-1}\ge0,
    \qquad
    \boldsymbol{0} \le S_h\one\le\one,
\]
we further obtain
\[
    p_{\min}^{n-1}\one
    \le
    p_{\min}^{n-1}S_h\one,
    \qquad
    p_{\max}^{n-1}S_h\one
    \le
    p_{\max}^{n-1}\one.
\]
Combining these inequalities with
\[
    \mathbf p^n = S_h\mathbf p^{n-1}
    \quad\Longrightarrow\quad
    p_{\min}^{n-1}\one
    \le \mathbf p^n
    \le p_{\max}^{n-1}\one.
\]
This is precisely the discrete maximum principle
\eqref{eq:cellwise-dmp}.
\end{proof}

\begin{remark}
The positivity property extends directly to the full pressure update
\eqref{eq:pressure-propagation}. If
\(\boldsymbol p^{n-1}\ge0\) and \(\boldsymbol\ell^n\ge0\), then
\(
  \boldsymbol p^n
  =
  S_h\boldsymbol p^{n-1}
  +
  K_h^{-1}\boldsymbol\ell^n
  \ge0,
\)
because \(S_h\ge0\) and \(K_h^{-1}\ge0\).
In particular, if the mechanics load is time independent, then \(\boldsymbol\ell^n=\tau\boldsymbol s^n\), and hence a nonnegative source preserves nonnegative cellwise pressures.
\end{remark}

\section{Convergence analysis}
\label{sec:convergence}

This section establishes the error estimate for the numerical approximation \eqref{eq:discrete-weak-form}. We continue to use the quasi-uniform mesh family, discrete spaces, and structural assumptions introduced above. 

\subsection{Coupled projection}
\label{subsec:coupled-projection}

We define a spatial projection associated with the four-field Biot formulation, combining the mass-lumped mixed Darcy approximation \cite{brezzi2006error,hu2017nonconforming} with the rotation--displacement mechanics formulation \cite{boon2023mixed}.
For each fixed $t\in[0,T_f]$, let
\(
(\bar r_h,\bar{u}_h,
 \bar{v}_h,\bar p_h)
\in
R_h\times U_h\times V_h\times P_h
\)
be defined by
\begin{subequations}
\label{eq:coupled-projection}
\begin{align}
\mu^{-1}(\bar r_h,\rho_h)
-
(\bar{u}_h,\nabla\times \rho_h)
=
\mu^{-1}(r,\rho_h)
-
(u,\nabla\times \rho_h), &
\label{eq:coupled-projection-mech1}
\\
(\nabla\times\bar r_h,w_h)
+
(\lambda+2\mu)
(\nabla\cdot\bar{u}_h,
 \nabla\cdot w_h)
-
\alpha(\bar p_h,\nabla\cdot w_h) &
\nonumber \\
=
(\nabla\times r,w_h)
+
(\lambda+2\mu)
(\nabla\cdot u,
 \nabla\cdot w_h)
-
\alpha(p,\nabla\cdot w_h), &
\label{eq:coupled-projection-mech2}
\\
a_{v,h}^{L}(\bar{v}_h, z_h)
-
(\bar p_h,\nabla\cdot z_h)
=
K^{-1}(v, z_h)
-
(p,\nabla\cdot z_h), &
\label{eq:coupled-projection-darcy1}
\\
(\nabla\cdot\bar{v}_h,q_h)
=
(\nabla\cdot v,q_h), &
\label{eq:coupled-projection-darcy2}
\end{align}
\end{subequations}
for all \(
(\rho_h,w_h,
 z_h,q_h)
\in
R_h\times U_h\times V_h\times P_h
\). 
For simplicity, we assume sufficient spatial regularity, in particular \(u\in H^2(\Omega)\) and \(r,\nabla \times r,v,p\in H^1(\Omega)\).
The projection is well posed by the stability of the two subproblems. The Darcy equations first determine \((\bar v_h,\bar p_h)\), after which the mechanics equations determine \((\bar r_h,\bar u_h)\).

For the Darcy variables $(\bar{\boldsymbol v}_h,\bar p_h)\in V_h\times P_h$,
the geometric condition~\eqref{eq:uniform-well-centered} yields the mesh-uniform stability of the circumcentric lumped Darcy product established in Lemma~\ref{lem:lumped-equivalence}.
Under these geometric assumptions, the error analysis of the mass-lumped mixed Darcy approximation in \cite{brezzi2006error,hu2017nonconforming}, together with the standard stability of the $\mathrm{RT}_0$--$\mathrm{DG}_0$ pair, gives
%For the Darcy variables \((\bar v_h,\bar p_h)\in V_h\times P_h\), the stability of the mixed \(\mathrm{RT}_0\)--\(\mathrm{DG}_0\) pair, together with the consistency and error analysis of the circumcentric mass-lumped mixed Darcy approximation in \cite{brezzi2006error,hu2017nonconforming}, gives
\begin{equation}
\label{eq:darcy-projection-L2-estimate}
  \lVert v-\bar v_h\rVert
  +
  \lVert p-\bar p_h\rVert
  \le
  Ch (
    \lVert v\rVert_{H^1(\Omega)}
    +
    \lVert p\rVert_{H^1(\Omega)}
   ).
\end{equation}
%Let \(\Pi_h^p:L^2(\Omega)\to P_h\) denote the cellwise \(L^2\)-orthogonal projection. 
%Since \(\nabla\cdot\bar v_h\in P_h\), the second Darcy projection equation implies
%\(
%  \nabla\cdot\bar v_h
%  =
%  \Pi_h^p(\nabla\cdot v).
%\)
%Hence, %for \(v\in H^2(\Omega)\),
%\begin{equation}
%\label{eq:darcy-projection-div-estimate}
%  \lVert\nabla\cdot(v-\bar v_h)\rVert
%  \le
%  Ch\lVert v\rVert_{H^2(\Omega)}.
%\end{equation}
%Combining
%\eqref{eq:darcy-projection-L2-estimate} and
%\eqref{eq:darcy-projection-div-estimate} yields
%\begin{equation}
%\label{eq:darcy-projection-estimate}
%  \lVert v-\bar v_h\rVert_{H(\operatorname{div};\Omega)}
%  +
%  \lVert p-\bar p_h\rVert
%  \le
%  Ch (
%    \lVert v\rVert_{H^2(\Omega)}
%    +
%    \lVert p\rVert_{H^1(\Omega)}
%  ).
%\end{equation}

For the mechanics projection
\((\bar r_h,\bar u_h)\in R_h\times U_h\), subtracting the defining
equations shows that the projection error is driven only by
\(p-\bar p_h\). The uniform stability of the compatible
rotation--displacement pair \cite{boon2023mixed} therefore gives
\begin{align*}
&\lVert r-\bar r_h\rVert_{H(\nabla\times;\Omega)}
+\lVert u-\bar u_h\rVert_{H(\nabla\cdot;\Omega)}
\\
& \le
C\bigg(
  \inf_{\rho_h\in R_h}
  \lVert r-\rho_h\rVert_{H(\nabla\times;\Omega)}
  +
  \inf_{w_h\in U_h}
  \lVert u-w_h\rVert_{H(\nabla\cdot;\Omega)}
\bigg)
+
C\lVert p-\bar p_h\rVert.
\end{align*}
Combining the standard interpolation estimates with \eqref{eq:darcy-projection-L2-estimate} yields
\begin{align}
\label{eq:mechanics-projection-estimate}
&\lVert r-\bar r_h\rVert_{H(\nabla\times;\Omega)}
+\lVert u-\bar u_h\rVert_{H(\nabla\cdot;\Omega)}
\notag \\
& \le
Ch\big(
  \lVert r\rVert_{H^1(\Omega)}
  +
  \lVert \nabla \times r\rVert_{H^1(\Omega)}
  +
  \lVert u\rVert_{H^2(\Omega)}
  +
  \lVert v\rVert_{H^1(\Omega)}
  +
  \lVert p\rVert_{H^1(\Omega)}
\big).
\end{align}

\subsection{Approximation estimates}

The total error is decomposed into the projection error and the discrete error.
\[
  z(t_n)-z_h^n
  =
  \bigl(z(t_n)-\bar z_h(t_n)\bigr)
  +
  \bigl(\bar z_h(t_n)-z_h^n\bigr)
  =:\eta_z^n+\theta_z^n,
\]
for \(z\in\{r,u,v,p\}\).
The initial data are chosen by the coupled projection to satisfy the discrete mechanics and Darcy constraints.
\begin{equation}
\label{eq:projected-initialization}
(r_h^0,u_h^0,v_h^0,p_h^0) = (\bar{r}_h(0),\bar{u}_h(0),\bar{v}_h(0),\bar p_h(0)).
\end{equation}
The residuals arising from the spatial projection and the backward Euler time discretization are defined by
\begin{align*}
  R_p^n
  &:=
  \delta_t \bar p_h^n -\partial_t p(t_n),
  %\label{eq:pressure-residual}
  \qquad
  R_u^n
  :=
  \nabla\cdot \delta_t\bar{u}_h^n 
  -\nabla\cdot\partial_t u(t_n).
  %\label{eq:displacement-residual}
\end{align*}
Following the temporal-residual argument in
\cite{hu2017nonconforming}, and applying the projection
estimates to the time derivatives, we obtain
\begin{align}
\tau\sum_{k=1}^{n}
\left(
  \lVert R_p^k\rVert^2
 +\lVert R_u^k\rVert^2
\right)
&\le
Ch^2\int_0^{t_n}
\Bigl(
  \lVert\partial_t r\rVert_{H^1(\Omega)}^2
 +\lVert\nabla\times\partial_t r\rVert_{H^1(\Omega)}^2
 +\lVert\partial_t u\rVert_{H^2(\Omega)}^2
\notag\\
&\qquad\qquad
 +\lVert\partial_t v\rVert_{H^1(\Omega)}^2
 +\lVert\partial_t p\rVert_{H^1(\Omega)}^2
\Bigr)\,dt
\notag\\
&\quad
+C\tau^2\int_0^{t_n}
\left(
  \lVert\partial_{tt}p\rVert^2
 +\lVert\nabla\cdot\partial_{tt}u\rVert^2
\right)\,dt .
\label{eq:error-residual-estimate}
\end{align}

\begin{lemma}[Discrete error estimates]
\label{lem:reduced-discrete-error} 
Assume the projected initialization \eqref{eq:projected-initialization}
and the regularity required for \eqref{eq:error-residual-estimate}.
For \(0\le n\le N\), the discrete mechanics errors satisfy
\begin{equation}
\label{eq:discrete-mechanics-reduction}
  \theta_r^n= 0,
  \qquad
  \nabla\cdot\theta_u^n
  =
  \frac{\alpha}{\lambda+2\mu}\theta_p^n,
  \qquad
  \lVert\theta_u^n\rVert_{H(\nabla\cdot;\Omega)}
  \le C\lVert\theta_p^n\rVert.
\end{equation}
Moreover, for \(1\le n\le N\), 
\begin{align}
& \left(
  c_0+\frac{\alpha^2}{\lambda+2\mu}
\right)\max_{0\le k\le n}
\lVert\theta_p^k\rVert^2
+
\tau\sum_{k=1}^{n}
a_{v,h}^{L}
(\theta_v^k,\theta_v^k)
\le C(h^2+\tau^2),
\label{eq:discrete-error-energy}
\\
& \max_{0\le k \le n} a_{v,h}^{L}(\theta_v^k,\theta_v^k)
+
\left( c_0+\frac{\alpha^2}{\lambda+2\mu} \right)
\tau\sum_{k=1}^{n}   \lVert\delta_t\theta_p^k\rVert^2
  \le C(h^2+\tau^2).
  \label{eq:discrete-flux-energy}
\end{align}
Here, \(C\) is independent of \(h\) and \(\tau\).
\end{lemma}

\begin{proof}
%For \(n=0\), the assertions follow from \eqref{eq:projected-initialization}. 
%Let \(1\le n\le N\). 
For \(n=0\), \eqref{eq:discrete-mechanics-reduction} follows from the initial condition in \eqref{eq:projected-initialization}. 
For \(1\le n\le N\), subtracting the mechanics projection equations \eqref{eq:coupled-projection-mech1}--\eqref{eq:coupled-projection-mech2} from the corresponding mechanics equations at \(t=t_n\) gives
\begin{subequations}
\begin{align}
  \mu^{-1}(\theta_r^n,\rho_h)
  -(\theta_u^n,\nabla\times\rho_h)
  &=0,
  &&\forall\,\rho_h\in R_h, \label{eq:diff-projection1}
  \\
  (\nabla\times\theta_r^n,w_h)
  +(\lambda+2\mu)
  (\nabla\cdot\theta_u^n,
   \nabla\cdot w_h)
  &=
  \alpha(\theta_p^n,\nabla\cdot w_h),
  &&\forall\, w_h\in U_h.
\label{eq:diff-projection2}
\end{align}
\end{subequations}
Let \([\theta_r^n]\), \([\theta_u^n]\), and \([\theta_p^n]\) denote the corresponding coefficient vectors. From \eqref{eq:fully-discrete-matrix}, the two mechanic error equations are equivalently
\[
\mu^{-1}M_r[\theta_r^n]-C_h^\top[\theta_u^n]=0,
\qquad
C_h[\theta_r^n]+(\lambda+2\mu)D_u[\theta_u^n]
=
\alpha B_u^\top[\theta_p^n].
\]
Eliminating the rotation gives
\[
  A_{u,h}[\theta_u^n]
  =
  \alpha B_u^\top[\theta_p^n].
\]
Hence, the exact mechanics Schur identity
\eqref{eq:mechanics-schur} yields
\[
\begin{aligned}
  B_u[\theta_u^n]
  &=
  \alpha B_uA_{u,h}^{-1}B_u^\top[\theta_p^n]
  =
  \frac{\alpha}{\lambda+2\mu}M_p[\theta_p^n].
\end{aligned}
\]
By the definitions of \(B_u\) and \(M_p\), this coordinate identity is
equivalent to
\[
  \left(
    \nabla\cdot\theta_u^n
    -\frac{\alpha}{\lambda+2\mu}\theta_p^n,
    q_h
  \right)=0
  \qquad\forall\,q_h\in P_h.
\]
Since both terms belong to \(P_h\),
\begin{align}
  \nabla\cdot\theta_u^n
  =
  \frac{\alpha}{\lambda+2\mu}\theta_p^n.
  \label{divergence-identity}
\end{align}

Substituting the divergence identity \eqref{divergence-identity} into \eqref{eq:diff-projection2} gives
\[
  (\nabla\times\theta_r^n, w_h)=0
  \qquad\forall\, w_h\in U_h.
\]
Since \(\nabla\times R_h\subset U_h\), we take \(w_h=\nabla\times\theta_r^n\) to obtain
\(
\nabla\times\theta_r^n = 0.
\)
Taking \(\rho_h=\theta_r^n\) in \eqref{eq:diff-projection1} gives
\(
  \mu^{-1}\lVert\theta_r^n\rVert^2=0,
\)
which implies 
\begin{align}
    \theta_r^n= 0.
\end{align}

The first mechanics error equation \eqref{eq:diff-projection1} therefore implies
\[
(\theta_u^n,\nabla\times\rho_h)=0
\qquad\forall\,\rho_h\in R_h,
\]
leading to \(\theta_u^n\in W_h\). 
As \(\ker(\nabla\cdot|_{U_h})=\nabla\times R_h\) stated in Subsection~\ref{subsec:mechanics-capacity}, the \(L^2\)-orthogonal complement \(W_h\) also coincides with the \(H(\nabla\cdot;\Omega)\)-orthogonal complement of \(\ker(\nabla\cdot|_{U_h})\).
Hence, the uniform inf--sup condition \eqref{eq:uniform-inf-sup}, together with the standard characterization of the inf--sup constant on the orthogonal complement of the kernel \cite{duran2008mixed,arnold2010finite}, gives
\[
\|\theta_u^n\|_{H(\nabla\cdot;\Omega)}
\le
\beta_u^{-1}
\|\nabla\cdot\theta_u^n\|.
\]
Combining this with \eqref{divergence-identity} yields the last inequality
in \eqref{eq:discrete-mechanics-reduction}.

Next, combining the continuous and discrete Darcy and mass-balance
equations at \(t=t_n\) with
\eqref{eq:coupled-projection-darcy1}--\eqref{eq:coupled-projection-darcy2},
and using \eqref{eq:discrete-mechanics-reduction}, we obtain
\begin{subequations}
\label{eq:reduced-error-system}
\begin{align}
a_{v,h}^{L}(\theta_v^n, z_h)
-(\theta_p^n,\nabla\cdot z_h)
&=0,
\label{eq:reduced-error-darcy}
\\
\left(
c_0+\frac{\alpha^2}{\lambda+2\mu}
\right)
(\delta_t\theta_p^n,q_h)
+
(\nabla\cdot\theta_v^n,q_h)
&=
c_0(R_p^n,q_h)
+
\alpha(R_u^n,q_h),
\label{eq:reduced-error-mass}
\end{align}
\end{subequations}
Taking $z_h=\theta_v^n$ and $q_h=\theta_p^n$ in \eqref{eq:reduced-error-system} and summing the resulting equations gives
\begin{align}
\left(
  c_0+\frac{\alpha^2}{\lambda+2\mu}
\right)
(\delta_t\theta_p^n,\theta_p^n)
+
a_{v,h}^{L}
(\theta_v^n,\theta_v^n)
=
c_0(R_p^n,\theta_p^n)
+\alpha(R_u^n,\theta_p^n).
\label{preceding-estimate1}
\end{align}
We then apply the inf--sup condition \eqref{eq:uniform-inf-sup} and
\eqref{eq:reduced-error-darcy} to obtain
\[
\begin{aligned}
\beta_v\lVert \theta_p^n\rVert
\le
\sup_{0\ne z_h\in V_h}
\frac{(\theta_p^n,\nabla\cdot z_h)}
     {\lVert z_h\rVert_{H(\nabla\cdot;\Omega)}} 
=
\sup_{0\ne z_h\in V_h}
\frac{a_{v,h}^{L}(\theta_v^n,z_h)}
     {\lVert z_h\rVert _{H(\nabla\cdot;\Omega)}} 
\le
C
\left[
a_{v,h}^{L}(\theta_v^n,\theta_v^n)
\right]^{1/2},
\end{aligned}
\]
where the last inequality follows from the uniform equivalence \eqref{eq:lumped-equivalence} and
continuity \eqref{eq:lumped-continuity} of \(a_{v,h}^{L}\). Hence, Young's inequality gives
\[
\left|
c_0(R_p^n,\theta_p^n)
+
\alpha(R_u^n,\theta_p^n)
\right|
\le
\frac12
a_{v,h}^{L}(\theta_v^n,\theta_v^n)
+
C\left(
\lVert R_p^n\rVert^2
+
\lVert R_u^n\rVert^2
\right).
\]
The identity
\[
  (\delta_t\theta_p^n,\theta_p^n)
  =
  \frac{1}{2\tau}
  \left(
    \lVert\theta_p^n\rVert^2
   -\lVert\theta_p^{n-1}\rVert^2
   +\lVert\theta_p^n-\theta_p^{n-1}\rVert^2
  \right),
\]
together with the preceding estimate \eqref{preceding-estimate1}, after multiplication by
\(2\tau\) and summation from \(k=1\) to \(m\), \(1\le m\le n\), yields
\[
\begin{aligned}
&
\left(
c_0+\frac{\alpha^2}{\lambda+2\mu}
\right)
\|\theta_p^m\|^2
+
\tau\sum_{k=1}^{m}
a_{v,h}^{L}(\theta_v^k,\theta_v^k)
\le
C\tau\sum_{k=1}^{m}
\left(
\|R_p^k\|^2
+
\|R_u^k\|^2
\right),
\end{aligned}
\]
where we have used \(\theta_p^0=0\) from
\eqref{eq:projected-initialization} and dropped the nonnegative term.
Taking \(m=n\) gives the bound for the full flux sum, while maximizing the pressure term over \(0\le m\le n\) gives the corresponding \(L^\infty\)-in-time bound. 
Combining these estimates and applying \eqref{eq:error-residual-estimate} yields \eqref{eq:discrete-error-energy}.

For the second energy estimate, applying the backward difference operator \(\delta_t\) to \eqref{eq:reduced-error-darcy} and testing with \(z_h=\theta_v^k\) gives
\[
a_{v,h}^{L}(\delta_t\theta_v^k,\theta_v^k)
-
(\delta_t\theta_p^k,\nabla\cdot\theta_v^k)
=0.
\]
Since \(a_{v,h}^{L}(\cdot,\cdot)\) is symmetric and time independent, there holds
\[
a_{v,h}^{L}(\delta_t\theta_v^k,\theta_v^k)
=
\frac12
\delta_t a_{v,h}^{L}(\theta_v^k,\theta_v^k)
+
\frac{\tau}{2}
a_{v,h}^{L}(\delta_t\theta_v^k,\delta_t\theta_v^k).
\]
Hence, we obtain
\begin{align}
\label{error-estimation2}
(\delta_t\theta_p^k,\nabla\cdot\theta_v^k)
=
\frac12
\delta_t a_{v,h}^{L}(\theta_v^k,\theta_v^k)
+
\frac{\tau}{2}
a_{v,h}^{L}(\delta_t\theta_v^k,\delta_t\theta_v^k).
\end{align}

Testing \eqref{eq:reduced-error-mass} with
\(q_h=\delta_t\theta_p^k\) gives
\[
\left(c_0+\frac{\alpha^2}{\lambda+2\mu}\right)
\lVert\delta_t\theta_p^k\rVert^2
+
(\nabla\cdot\theta_v^k,\delta_t\theta_p^k)
=
c_0(R_p^k,\delta_t\theta_p^k)
+\alpha(R_u^k,\delta_t\theta_p^k).
\]
Combining the identity \eqref{error-estimation2} and applying Young's inequality, we obtain
\[
\begin{aligned}
&\frac12 \delta_t a_{v,h}^{L}(\theta_v^k,\theta_v^k)
+\frac{\tau}{2} a_{v,h}^{L}(\delta_t\theta_v^k,\delta_t\theta_v^k)
+\left(c_0+\frac{\alpha^2}{\lambda+2\mu}\right)
\lVert\delta_t\theta_p^k\rVert^2
\\
& \le
C\left(
\lVert R_p^k\rVert^2
+\lVert R_u^k\rVert^2
\right)
+\frac12
\left(c_0+\frac{\alpha^2}{\lambda+2\mu}\right)
\lVert\delta_t\theta_p^k\rVert^2.
\end{aligned}
\]
Multiplying by \(2\tau\), summing from \(k=1\) to \(m\), \(1\le m\le n\), we derive
\[
\begin{aligned}
&a_{v,h}^{L}(\theta_v^m,\theta_v^m)
+
\left(c_0+\frac{\alpha^2}{\lambda+2\mu}\right)
\tau\sum_{k=1}^{m}
\lVert\delta_t\theta_p^k\rVert^2
\le
C\tau\sum_{k=1}^{m}
\left(
\lVert R_p^k\rVert^2
+\lVert R_u^k\rVert^2
\right),
\end{aligned}
\]
where we have used \(\theta_v^0=0\) from
\eqref{eq:projected-initialization} and dropped the nonnegative term. 
Taking \(m=n\) and maximizing over \(0\le m\le n\), then applying \eqref{eq:error-residual-estimate}, yields \eqref{eq:discrete-flux-energy}.
\end{proof}

\begin{theorem}[First-order convergence]
\label{thm:first-order-convergence}
Assume the conditions of Lemma~\ref{lem:reduced-discrete-error} and the projection estimates \eqref{eq:darcy-projection-L2-estimate}, \eqref{eq:mechanics-projection-estimate} hold.
Then, for every \(1\le n\le N\),
\begin{align*}
%\label{eq:fully-discrete-error-estimate}
\max\left\{
\lVert r(t_n)-r_h^n\rVert_{H(\nabla\times;\Omega)},
\lVert u(t_n)-u_h^n\rVert_{H(\nabla\cdot;\Omega)},
\lVert v(t_n)-v_h^n\rVert,
\lVert p(t_n)-p_h^n\rVert
\right\} 
\lesssim h+\tau.
\end{align*}
\end{theorem}

\begin{proof}
The result follows from the error decomposition
\[
z(t_n)-z_h^n
=
\eta_z^n+\theta_z^n,
\qquad
z\in\{r,u,v,p\},
\]
the projection estimates
\eqref{eq:darcy-projection-L2-estimate} and
\eqref{eq:mechanics-projection-estimate},
the norm equivalence \eqref{eq:lumped-equivalence}, and
Lemma~\ref{lem:reduced-discrete-error}.
\end{proof}

%\section{Convergence analysis}
%\label{sec:convergence}

%This section establishes the error estimate for the numerical approximation \eqref{eq:discrete-weak-form}. We retain the quasi-uniform mesh family, discrete spaces, and structural assumptions introduced above. 

\section{Numerical experiments}
\label{sec:numerics}

This section presents numerical experiments to illustrate the theoretical properties of the proposed method.
We first examine the discrete pressure monotonicity and the predicted first-order convergence rates on a uniformly acute mesh family in two dimensions. 
A Barry--Mercer problem is then considered to assess the preservation of nonnegative cellwise pressures under a localized approximation of point-source injection.
Finally, a three-dimensional test on a well-centered tetrahedral mesh demonstrates that the same pressure monotonicity mechanism remains effective in three dimensions.
All computations are carried out using FEniCSx \cite{BarattaEtal2023}.
For the convergence, we report the following relative errors:
\begin{align*}
& \widehat E_r
:=
\frac{\lVert r- r_h\rVert_{H(\nabla\times;\Omega)}}
{\lVert r\rVert_{H(\nabla\times;\Omega)}},
\
\widehat E_u
:=
\frac{\lVert u - u_h\rVert_{H(\nabla\cdot;\Omega)}}
{\lVert u\rVert_{H(\nabla\cdot;\Omega)}}, 
\
\widehat E_v
:=
\frac{\lVert v- v_h\rVert}
{\lVert v\rVert},
\
\widehat E_p
:=
\frac{\lVert p -p_h\rVert}
{\lVert p\rVert}.
\end{align*}

\subsection{Two-dimensional monotonicity and convergence}
\label{subsec:numerics-two-dimensional}

In this test, we employ the manufactured solutions on the unit square $\Omega=(0,1)^2$ to examine the pressure monotonicity of the scheme and the convergence behavior predicted by Theorem~\ref{thm:first-order-convergence}.
For the boundary decomposition, we set
\(
\Gamma_u:=\{1\}\times(0,1),
\)
and
\(
\Gamma_r:=\partial\Omega\setminus\overline{\Gamma_u}.
\) 
The auxiliary functions and the corresponding exact solution are given by
\begin{align*}
& \phi(x,y)=\sin^4(\pi x)\sin^4(\pi y),
\qquad
\psi(x,y)=\sin^2(\pi x)\sin^2(\pi y),
\\
& u = A(t) (\nabla\phi-0.3\nabla\times\phi ), 
\quad
r = \mu\nabla\times u, 
\quad
p = A(t) \psi, 
\quad
v = -K\nabla p.
\end{align*}
The prescribed solution satisfies the homogeneous boundary conditions.
The forcing terms are obtained by substitution into the continuous system.
For the physical parameters, we consider
\[
\mu=\lambda=\alpha=c_0=K=1.
\]

\paragraph{Mesh setting} 
We use the acute triangulation \(\mathcal T_0\) of the unit square \cite{eppstein2004tiling}, shown on the left of Figure~\ref{fig:acute-square-refinement}. 
The mesh consists of 14 triangles, has maximum element diameter \(h_0=0.54\), and is uniformly acute, thereby ensuring the positivity of the local circumcentric distances used in the Darcy discretization.
A mesh family \(\{\mathcal T_\ell\}_{\ell\ge0}\) is generated by uniform red refinement, in which each triangle is subdivided into four similar triangles by joining the edge midpoints. 
The first refinement \(\mathcal T_1\) is shown on the right of Figure~\ref{fig:acute-square-refinement}. Hence, the family remains uniformly acute and satisfies \(h_\ell=2^{-\ell}h_0\).
For this self-similar mesh family, a direct computation gives
\[
  \min_{T,F}\frac{d_{T,F}^{s}}{h_T}
  =
  0.1624598481
\]
at every level. Thus, the mesh family satisfies both the positive-weight condition of Section~\ref{sec:monotonicity} and the quantitative well-centeredness condition of Section~\ref{sec:convergence}.
\begin{figure}[H]
  \centering
  \begin{tikzpicture}[
      x=3.0cm,y=3.0cm,
      mesh/.style={draw=blue!65!black,line width=0.55pt},
      boundary/.style={draw=black,line width=1.0pt},
      vertex/.style={circle,fill=black,inner sep=1.1pt}]
    \coordinate (A) at (0,0);
    \coordinate (B) at (1,0);
    \coordinate (C) at (1,1);
    \coordinate (D) at (0,1);
    \coordinate (E) at (0.4596495484,1);
    \coordinate (F) at (1,0.4596495484);
    \coordinate (G) at (0.5403504516,0);
    \coordinate (H) at (0,0.5403504516);
    \coordinate (I) at (0.3044993700,0.6955006300);
    \coordinate (J) at (0.6420395219,0.6420395219);
    \coordinate (K) at (0.3579604781,0.3579604781);
    \coordinate (L) at (0.6955006300,0.3044993700);
    \foreach \a/\b/\c in
      {D/I/E,D/H/I,C/E/J,C/J/F,A/G/K,A/K/H,B/L/G,B/F/L,
       I/J/E,I/K/J,I/H/K,L/F/J,L/J/K,L/K/G}
      {\draw[mesh] (\a)--(\b)--(\c)--cycle;}
    \draw[boundary] (A)--(B)--(C)--(D)--cycle;
    %\node[font=\small] at (0.5,-0.16)
    %  {(a) Fourteen-triangle acute triangulation \(\mathcal T_0\)};
  \end{tikzpicture}
  \hspace{0.6cm}
  \begin{tikzpicture}[
      x=3.0cm,y=3.0cm,
      coarse/.style={draw=black!45,line width=0.45pt},
      refined/.style={draw=red!75!black,line width=0.55pt},
      boundary/.style={draw=black,line width=1.0pt}]
    \coordinate (A) at (0,0);
    \coordinate (B) at (1,0);
    \coordinate (C) at (1,1);
    \coordinate (D) at (0,1);
    \coordinate (E) at (0.4596495484,1);
    \coordinate (F) at (1,0.4596495484);
    \coordinate (G) at (0.5403504516,0);
    \coordinate (H) at (0,0.5403504516);
    \coordinate (I) at (0.3044993700,0.6955006300);
    \coordinate (J) at (0.6420395219,0.6420395219);
    \coordinate (K) at (0.3579604781,0.3579604781);
    \coordinate (L) at (0.6955006300,0.3044993700);
    \foreach \a/\b/\c in
      {D/I/E,D/H/I,C/E/J,C/J/F,A/G/K,A/K/H,B/L/G,B/F/L,
       I/J/E,I/K/J,I/H/K,L/F/J,L/J/K,L/K/G}
      {
        \draw[coarse] (\a)--(\b)--(\c)--cycle;
        \draw[refined]
          ($(\a)!0.5!(\b)$)--
          ($(\b)!0.5!(\c)$)--
          ($(\c)!0.5!(\a)$)--cycle;
      }
    \draw[boundary] (A)--(B)--(C)--(D)--cycle;
    %\node[font=\small] at (0.5,-0.16) {(b) One red refinement \(\mathcal T_1\)};
  \end{tikzpicture}
  \caption{Acute mesh family used in the two-dimensional numerical tests.
  Left: the initial 14-triangle mesh \(\mathcal T_0\).
  Right: its first uniform red refinement \(\mathcal T_1\).}
  \label{fig:acute-square-refinement}
\end{figure}
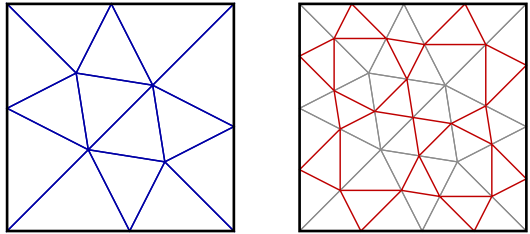

\paragraph{Monotonicity test}
We first examine the condensed pressure matrices with \(T_f= \tau=0.1\).
Recall from \eqref{eq:pressure-propagation} that
\(
S_h=K_h^{-1}H_h.
\)
For the refinement levels \(\ell=0,\ldots,3\), the relative defect in the
mechanics Schur identity
\(
B_uA_{u,h}^{-1}B_u^\top=(\lambda+2\mu)^{-1}M_p
\)
remains below \(2.9\times10^{-15}\), confirming \eqref{eq:mechanics-schur} to the machine precision.
The maximum off-diagonal entry of \(K_h\) is below \(5.8\times10^{-18}\), consistent with the \(Z\)-matrix property up to roundoff, while both \(K_h^{-1}\) and \(S_h\) are entrywise nonnegative.
As shown in Table~\ref{tab:two-d-matrix-audit}, \(K_h\) is positive definite and
\(
0<S_h\one<\one.
\)
These results confirm the Stieltjes structure of \(K_h\) and the row-substochastic property of \(S_h\) established in Section~\ref{sec:monotonicity}.

\begin{table}[H]
  \centering
  \caption{Pressure monotonicity check on the acute mesh family for
  \(\tau=0.1\).}
  \label{tab:two-d-matrix-audit}
  \small
  \setlength{\tabcolsep}{4pt}
  \begin{tabular}{@{}ccccc@{}}
    \hline
    \(\ell\) &
    \(\lambda_{\min}(K_h)\) &
    \(\min_{i,j}(K_h^{-1})_{ij}\) &
    \(\min_{i,j}(S_h)_{ij}\) &
    \([\min S_h\one,\max S_h\one]\) \\
    \hline
    0 & \(1.889\times10^{-1}\) & \(2.765\times10^{-3}\) &
        \(2.580\times10^{-4}\) & \([0.1722,0.5434]\) \\
    1 & \(4.986\times10^{-2}\) & \(1.702\times10^{-4}\) &
        \(3.970\times10^{-6}\) & \([0.0650,0.5389]\) \\
    2 & \(1.258\times10^{-2}\) & \(1.027\times10^{-5}\) &
        \(5.988\times10^{-8}\) & \([0.0222,0.5517]\) \\
    3 & \(3.153\times10^{-3}\) & \(6.303\times10^{-7}\) &
        \(9.189\times10^{-10}\) & \([0.0071,0.5550]\) \\
    \hline
  \end{tabular}
\end{table}

The same matrix audit is repeated for \(\tau=10^{-3},10^{-5},10^{-7},10^{-9}\).
In all cases, the computed pressure remains nonnegative up to roundoff, indicating that reducing the time step does not introduce spurious negative cellwise pressures.

\paragraph{Space--time convergence}
We refine space and time simultaneously with \(A(t)=e^t\) and \(T_f=1.0\).
We use the mesh levels \(\ell=0,1,2,3\) with \(\tau_\ell=2^{-(\ell+2)}\), so that both the mesh size and the time step decrease by a factor of two at each refinement.
The observed orders in Table~\ref{tab:two-d-spacetime-errors} are computed between successive coupled refinements and confirm the predicted first-order space-time convergence.

\begin{table}[H]
  \centering
  \caption{Space--time errors and observed orders under coupled refinement.}
  \label{tab:two-d-spacetime-errors}
  \small
  \setlength{\tabcolsep}{3pt}
  %\resizebox{\textwidth}{!}{%
  \begin{tabular}{@{}ccccccccc@{}}
    \hline
    \(\ell\)
    & \(\widehat E_r\) & Order
    & \(\widehat E_u\) & Order
    & \(\widehat E_v\) & Order
    & \(\widehat E_p\) & Order \\
    \hline
    0
    & \(9.644\times10^{-1}\) & --
    & \(5.420\times10^{-1}\) & --
    & \(5.712\times10^{-1}\) & --
    & \(4.622\times10^{-1}\) & -- \\
    1
    & \(5.764\times10^{-1}\) & 0.742
    & \(4.022\times10^{-1}\) & 0.430
    & \(3.138\times10^{-1}\) & 0.864
    & \(2.279\times10^{-1}\) & 1.020 \\
    2
    & \(3.117\times10^{-1}\) & 0.887
    & \(2.101\times10^{-1}\) & 0.937
    & \(1.599\times10^{-1}\) & 0.973
    & \(1.141\times10^{-1}\) & 0.998 \\
    3
    & \(1.585\times10^{-1}\) & 0.976
    & \(1.066\times10^{-1}\) & 0.979
    & \(8.039\times10^{-2}\) & 0.992
    & \(5.711\times10^{-2}\) & 0.998 \\
    \hline
  \end{tabular}%}
\end{table}

\subsection{Barry--Mercer's problem}
\label{subsec:numerics-barry-mercer}

In this section, we consider the Barry--Mercer problem
\cite{barry1999exact} on \(\Omega=(0,1)^2\) with a point-source
injection centered at \(\boldsymbol x_0=(1/4,1/4)\).
We use this example to examine the accuracy of the proposed method and,
in particular, the preservation of nonnegative pressures.
We consider parameters \(\alpha=1.0\) and \(c_0=0\), define \(\omega:=\lambda+2\mu \) and prescribe the body force and the fluid source by
\begin{equation}
\label{original-dirac}
f(\boldsymbol x,t) = 0 , \qquad s(\boldsymbol x,t)
=
2K \omega \,
\delta_{\boldsymbol x_0}(\boldsymbol x)
\sin\bigl(K \omega t\bigr).
\end{equation}
We take
\(
\Gamma_r=\varnothing
\) and \(
\Gamma_u=\partial\Omega,
\)
so that the boundary conditions reduce to
\begin{equation*}
%\label{eq:barry-mercer-boundary}
\nu\times u=0,
\qquad
(\lambda+2\mu)\nabla\cdot u-p=0,
\qquad
p=0
\qquad\text{on }\partial\Omega.
\end{equation*}
The system is initialized at the zero state.
The analytical pressure is then given by
\begin{align*}
& p(x,y,t)=8 \omega
\sum_{m,n=1}^{\infty}
a_{mn}F_{mn}(t)
\sin(m\pi x)\sin(n\pi y),
\\
& u(x,y,t)=-8
\sum_{m,n=1}^{\infty}
\frac{a_{mn}F_{mn}(t)}{\gamma_{mn}}
\begin{pmatrix}
m\pi\cos(m\pi x)\sin(n\pi y)
\\ %[2mm]
n\pi\sin(m\pi x)\cos(n\pi y)
\end{pmatrix},
\end{align*}
where 
\begin{align*}
& a_{mn}
=
\sin\left(\frac{m\pi}{4}\right)
\sin\left(\frac{n\pi}{4}\right),
\qquad
\gamma_{mn}
=
\pi^2(m^2+n^2),
\\
&F_{mn}(t)
=
\frac{
\gamma_{mn}\sin\bigl(K\omega t\bigr)
-\cos\bigl(K \omega t\bigr)
+\exp\!\bigl[-\gamma_{mn}K \omega t\bigr]
}{
\gamma_{mn}^2+1
}.
\end{align*}
The remaining variables satisfy
\begin{equation*}
%\label{eq:barry-mercer-exact-relations}
v=-K\nabla p,
\qquad
r=\mu\nabla\times u=0,
\qquad
\nabla\cdot u=\frac{p}{\lambda+2\mu}.
\end{equation*}

\paragraph{Mesh and source regularization}
Starting from the level-two acute mesh \(\mathcal T_2\) introduced in  Subsection~\ref{subsec:numerics-two-dimensional}, we apply a smooth interior perturbation while keeping the boundary fixed and denote the resulting mesh by \(\widetilde{\mathcal T}_2\).
The perturbation is chosen so that \(\boldsymbol x_0\) is the barycenter of a unique cell \(\widetilde{\mathcal T}_2^\star\).
The mesh \(\widetilde{\mathcal T}_2\) remains acute, with angles between
\(37.93^\circ\) and \(81.60^\circ\), and satisfies
\[
\min_{T\in \widetilde{\mathcal{\mathcal T}}_2}
\min_{F\subset\partial T}
\frac{d_{T,F}^{s}}{h_T}
=
0.07383.
\]
For \(\ell>2\), the meshes are generated by uniform red refinement. The angle bounds and the well-centeredness are preserved.
Moreover, \(\boldsymbol x_0\) remains the barycenter of a unique cell \(\widetilde{T}_\ell^\star\) at every refinement level.
Since \(P_h=\mathrm{DG}_0\) and \(\boldsymbol x_0\) lies in the interior
of \(\widetilde T_\ell^\star\), we represent the Dirac distribution
\(\delta_{\boldsymbol x_0}\) in \eqref{original-dirac} as follows.
\begin{equation}
\label{eq:barry-mercer-discrete-dirac}
\delta_{h,\boldsymbol x_0}(\boldsymbol x)
:=
\frac{1}{|\widetilde T_\ell^\star|}
\chi_{\widetilde T_\ell^\star}(\boldsymbol x),
\qquad
\int_\Omega \delta_{h,\boldsymbol x_0}\,\mathrm dx=1.
\end{equation}

\paragraph{Comparison of the Darcy products}

In this test, we consider the remaining physical parameters as \(\mu=0.4\), \(\lambda=0.2\), and \(K=1.0\).
We compare the proposed method, based on the circumcentric product
\(a_{v,h}^{L}\) defined in \eqref{eq:mass-lumped-darcy-form}, with the
consistent product
\[
a_{v,h}^{C}( v_h, z_h)
:=
K^{-1}( v_h, z_h).
\]
All other spaces, forms, data, and time-stepping terms are kept identical.
The numerical solutions are compared with the analytical series at the final time
\(
T_f=\pi/(K\omega).
\)
For the convergence study, the time steps are chosen such that \(\tau_\ell \simeq h_\ell^2\), so that temporal errors are negligible relative to the spatial errors.

Table~\ref{tab:barry-mercer-exact-errors} shows that the errors for both the circumcentric and consistent Darcy products decrease at approximately first order in displacement, flux, and pressure.
The consistent product gives slightly smaller errors, while the circumcentric product exhibits essentially the same convergence behavior.
The results also show that the circumcentric scheme retains comparable accuracy.
Since the exact rotation vanishes, its relative error is omitted.

\begin{table}[H]
  \centering
  \caption{Relative errors against the Barry--Mercer analytical solutions at \(T_f=\pi/(K\omega)\) for the circumcentric and consistent Darcy products.}
  \label{tab:barry-mercer-exact-errors}
  \small
  \begin{tabular}{@{}ccccc@{}}
    \hline
    Product & \(\ell\)
    & \(\widehat E_u\)
    & \(\widehat E_v\)
    & \(\widehat E_p\) \\
    \hline
    \(a_{v,h}^{L}\) & 2
    & \(1.1750\times10^{-1}\)
    & \(1.4211\times10^{-1}\)
    & \(1.1765\times10^{-1}\) \\
    \(a_{v,h}^{L}\) & 3
    & \(6.0373\times10^{-2}\;(0.961)\)
    & \(7.2354\times10^{-2}\;(0.974)\)
    & \(6.0418\times10^{-2}\;(0.961)\) \\
    \hline
    \(a_{v,h}^{C}\) & 2
    & \(1.0929\times10^{-1}\)
    & \(1.3622\times10^{-1}\)
    & \(1.0921\times10^{-1}\) \\
    \(a_{v,h}^{C}\) & 3
    & \(5.4745\times10^{-2}\;(0.997)\)
    & \(6.8590\times10^{-2}\;(0.990)\)
    & \(5.4708\times10^{-2}\;(0.997)\) \\
    \hline
  \end{tabular}
\end{table}

\paragraph{Pressure positivity}
For the positivity test, we use the low-permeability regime with \(\mu=1.0\), \(\lambda=10^4\), and \(K=10^{-6}\).
Figure~\ref{fig:barry-mercer-pressure-comparison} shows the pressure after the first positive injection step on \(\mathcal T_3\) at \(t_1=5.387\times10^{-3}\). 
Cells with significant negative pressure are outlined in blue.
For the circumcentric scheme, the pressure remains nonnegative and is concentrated around the injection point. 
In contrast, the consistent product produces a region of negative pressure surrounding the positive core, with 105 negative cells at the first step. 
This comparison highlights the positivity-preserving behavior of the circumcentric Darcy discretization, with no spurious negative cellwise pressures observed in this test.

\begin{figure}[htbp]
  \centering
  \includegraphics
  [
    width=0.8\textwidth,
    trim=0cm 0.2cm 0cm 0cm,
    clip
  ]
    {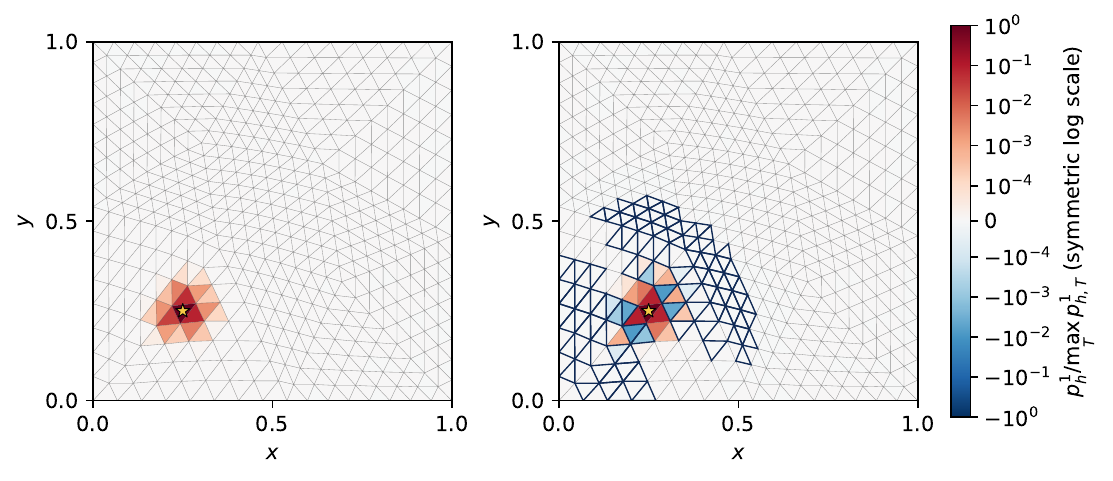}
    \caption{Normalized pressure after the first positive injection step on
    \(\mathcal T_3\): circumcentric lumping (left) and the consistent
    \(\mathrm{RT}_0\) Darcy product (right).}
  \label{fig:barry-mercer-pressure-comparison}
\end{figure}

\subsection{Three-dimensional tests on a Sommerville mesh family}
\label{subsec:numerics-three-dimensional}

In this test, we consider the following special three-dimensional domain:
\begin{align*}
  B=
  \begin{pmatrix}
    0 & \frac12 & \frac12 \\[0.3mm]
    0 & \frac{\sqrt3}{2} & -\frac{\sqrt3}{2} \\[0.3mm]
    \frac{3}{2\sqrt2} & \frac1{\sqrt2} & -\frac1{2\sqrt2}
  \end{pmatrix},
  \qquad
  \Omega_S=\{B\boldsymbol\xi:\boldsymbol\xi\in(0,1)^3\},
\end{align*}
which defines an oblique parallelepiped.
For \(\boldsymbol x\in\Omega_S\), let
\(
  \boldsymbol\xi
  =(\xi_1,\xi_2,\xi_3)^\top
  :=B^{-1}\boldsymbol x\in(0,1)^3
\)
denote the corresponding reference coordinates. 
A well-centered Sommerville mesh family
\cite{korotov2014red} is constructed on \(\Omega_S\).
Let \(\mathcal T_1\) be the six-tetrahedron Sommerville partition of \(\Omega_S\), and let \(\{\mathcal T_N\}_{N\ge1}\) denote the self-similar refinement family. Then
\begin{align*}
  h_N=\frac{1}{N}\sqrt{\frac32},
  \qquad
  \min_{T,F}\frac{d_{T,F}^{s}}{h_T}
  =\frac{1}{4\sqrt2}.
\end{align*}
so the family is uniformly well centered. Under \(N\mapsto2N\), each tetrahedron is subdivided into eight similar children, as illustrated in Figure~\ref{fig:3d-mesh}.

\begin{figure}[htbp]
  \centering
  \includegraphics[
    width=0.42\textwidth,
    trim=1.45cm 1.48cm 1.45cm 2.2cm,
    clip
  ]{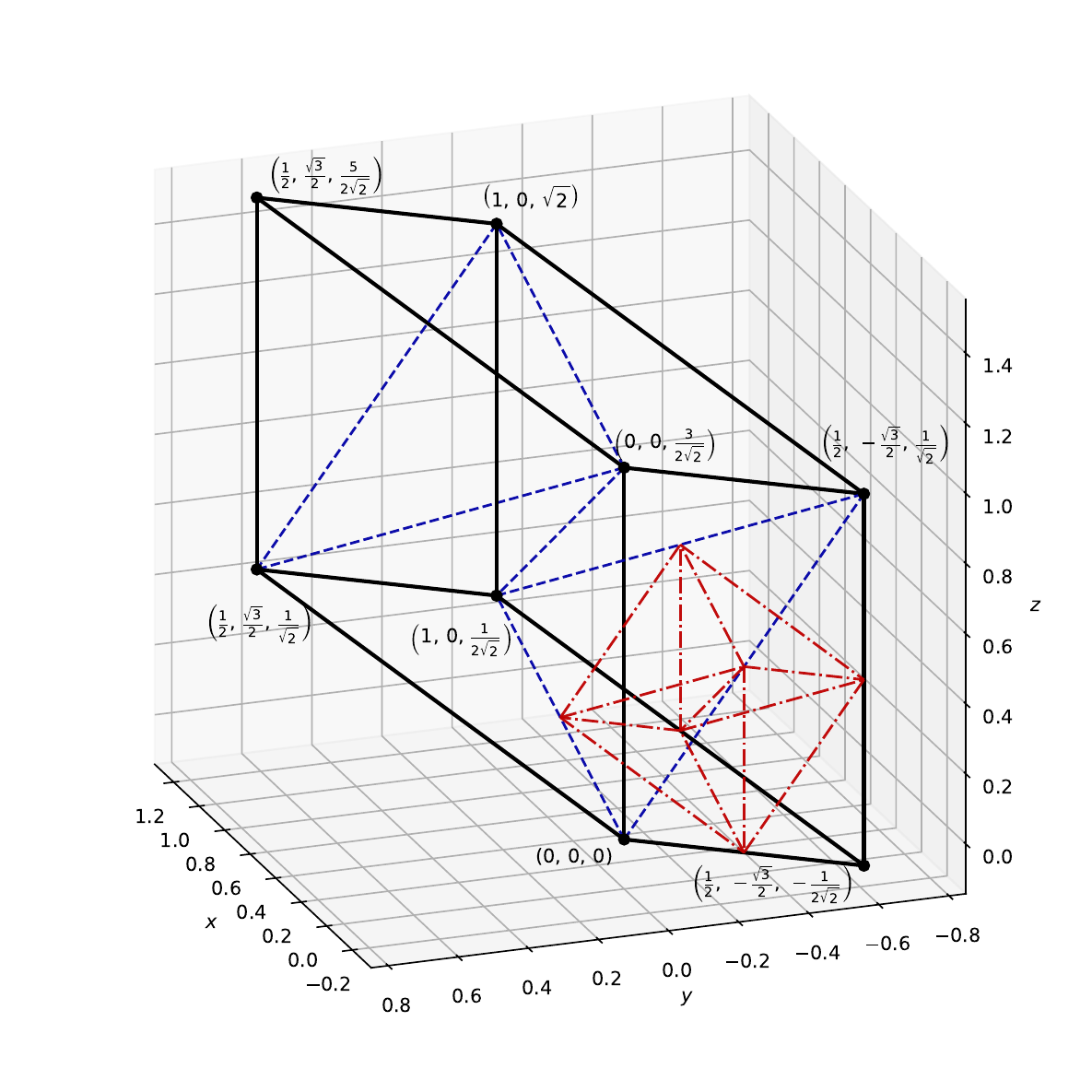}
  \caption{Sommerville mesh family used in the three-dimensional numerical
  tests. The black solid lines show the boundary of \(\Omega_S\), the blue
  dashed lines the six-tetrahedron coarse decomposition, and the red
  dash-dotted lines a representative one-to-eight refinement.}
  \label{fig:3d-mesh}
\end{figure}

\paragraph{Monotonicity test}
We examine the pressure monotonicity for different physical parameters on the fixed mesh \(\mathcal T_4\), which consists of 384 tetrahedra, and take \(\tau=10^{-3}\) throughout.
For the mechanics boundary decomposition, we define \(\Gamma_r\) as the image under \(B\) of the face \(\xi_3=0\) of the reference cube:
\(
  \Gamma_r
  =
  \{
  B(\xi_1,\xi_2,0)^\top:
  0<\xi_1,\xi_2<1
  \}
\),
\(
  \Gamma_u
  =
  \partial\Omega_S\setminus\overline{\Gamma_r}.
\)
We set \(\mu=\alpha=1\) and consider \(\lambda\in\{1,10^4\}\), \(c_0\in\{0,1\}\), and \(K\in\{1,10^{-6}\}\), giving eight parameter combinations.
Using the same mesh, boundary conditions, and time step in all cases isolates the effect of the material parameters.

As shown in Table~\ref{tab:three-d-matrix-audit}, \(K_h\) is positive definite for all eight parameter combinations, while \(K_h^{-1}\) and \(S_h\) are entrywise nonnegative up to roundoff. The row sums of \(S_h\) are also bounded by one up to roundoff. Accordingly, \(S_h\) is row-substochastic up to roundoff, and the same pressure-monotonicity mechanism is observed in three dimensions.

\begin{table}[H]
  \centering
  \caption{Pressure monotonicity check on the three-dimensional
  Sommerville mesh \(\mathcal T_4\).}
  \label{tab:three-d-matrix-audit}
  \small
  \setlength{\tabcolsep}{3pt}
  \begin{tabular}{@{}ccccccc@{}}
    \hline
    \(\lambda\) &
    \(c_0\) &
    \(K\) &
    \(\lambda_{\min}(K_h)\) &
    \(\min_{i,j}(K_h^{-1})_{ij}\) &
    \(\min_{i,j}(S_h)_{ij}\) &
    \([\min S_h\one,\max S_h\one]\) \\
    \hline
    \(1\) &
    \(1\) &
    \(1\) &
    \(3.275\times10^{-3}\) &
    \(8.399\times10^{-31}\) &
    \(2.679\times10^{-33}\) &
    \([0.64550,0.99998]\) \\

    \(1\) &
    \(1\) &
    \(10^{-6}\) &
    \(3.189\times10^{-3}\) &
    \(1.294\times10^{-229}\) &
    \(4.128\times10^{-232}\) &
    \([0.99999,1.00000]\) \\

    \(1\) &
    \(0\) &
    \(1\) &
    \(8.826\times10^{-4}\) &
    \(1.882\times10^{-18}\) &
    \(1.501\times10^{-21}\) &
    \([0.32700,0.99630]\) \\

    \(1\) &
    \(0\) &
    \(10^{-6}\) &
    \(7.974\times10^{-4}\) &
    \(1.528\times10^{-208}\) &
    \(1.218\times10^{-211}\) &
    \([0.99999,1.00000]\) \\

    \(10^4\) &
    \(1\) &
    \(1\) &
    \(2.478\times10^{-3}\) &
    \(9.385\times10^{-28}\) &
    \(2.245\times10^{-30}\) &
    \([0.58010,0.99994]\) \\

    \(10^4\) &
    \(1\) &
    \(10^{-6}\) &
    \(2.392\times10^{-3}\) &
    \(3.043\times10^{-225}\) &
    \(7.280\times10^{-228}\) &
    \([0.99999,1.00000]\) \\

    \(10^4\) &
    \(0\) &
    \(1\) &
    \(8.550\times10^{-5}\) &
    \(5.865\times10^{-8}\) &
    \(1.403\times10^{-14}\) &
    \([0.00017,0.00442]\) \\

    \(10^4\) &
    \(0\) &
    \(10^{-6}\) &
    \(2.392\times10^{-7}\) &
    \(2.516\times10^{-85}\) &
    \(6.016\times10^{-92}\) &
    \([0.99240,1.00000]\) \\
    \hline
  \end{tabular}
\end{table}

\section{Conclusions}
\label{sec:conclusion}

In this work, we discuss pressure monotonicity  for the lowest-order rotation-based Biot discretization and identify two compatible structures governing the discrete pressure evolution.
Relative exactness yields a positive diagonal mechanics capacity, while circumcentric mass lumping gives a weighted cell-graph Laplacian for the Darcy contribution. 
Under positive facet weights, the resulting pressure matrix is Stieltjes, which leads to a nonnegative pressure propagator and cellwise discrete maximum principles. 
We also established the uniform norm equivalence of the mass-lumped Darcy product and combined it with the corresponding mixed approximation estimates to derive a first-order convergence result.
The numerical experiments confirm the predicted convergence behavior and demonstrate the pressure-monotonicity and positivity-preserving properties in both two and three dimensions.

%We also established uniform norm equivalence and approximation properties for the mass-lumped Darcy product, leading to first-order error estimates. Numerical experiments confirm the predicted convergence behavior and the pressure-monotonicity properties in both two and three dimensions.

%\appendix
%\section{An example appendix} 
%\lipsum[71]

%\section*{Acknowledgments}
%We would like to acknowledge the assistance of volunteers in putting
%together this example manuscript and supplement.

\bibliographystyle{siamplain}
\bibliography{references}

\end{document}